\documentclass{amsart}
\usepackage[utf8]{inputenc}
\usepackage{amsmath}
\usepackage{amsfonts}
\usepackage{graphicx} % Required for inserting images
\usepackage{mathtools}
\usepackage{xcolor}
\usepackage[pdftex]{hyperref}

\begin{document}
\title[Cahn-Hilliard Equations on Lattices]{Cahn-Hilliard Equations on Lattices: Dynamic Transitions and Pattern Formations}
\author[Grossman]{Jared Grossman}
\address[Grossman]{Department of Mathematics and Statistics, Boston University, Boston, MA 02215}
\email{jaredg@bu.edu}
\author[Halloran]{Evan Halloran}
\address[Halloran]{Department of Mathematics, Indiana University, Bloomington, IN 47405}
\email{ehallor@iu.edu}
\author[Wang]{Shouhong Wang}
\address[Wang]{Department of Mathematics, Indiana University, Bloomington, IN 47405}
\email{showang@iu.edu}

\date{July 26, 2024}

\newcommand{\Q}{\mathbb{Q}}
\newcommand{\R}{\mathbb{R}}
\newcommand{\Z}{\mathbb{Z}}
\newcommand{\N}{\mathbb{N}}
\newcommand{\cd}{\cdot}
\newcommand{\cds}{\cdots}
\newcommand{\C}{\mathbb{C}}
\newcommand{\CP}{C(\R\setminus \Z ; \C)}
\newcommand{\conj}{\overline}
\newtheorem{theorem}{Theorem}[section]
\newtheorem{definition}[theorem]{Definition}
 \newcommand{\comment}[1]{\textcolor{red}{#1}}
\newtheorem{remark}{Remark}[section]

\def\bd{\begin{defi}}
\def\ed{\end{defi}}
\newtheorem{thm}{Theorem}[section]
\newtheorem{lem}{Lemma}[section]
\newtheorem{defi}{Definition}[section]
\newtheorem{ex}{Example}[section]
\newtheorem{prop}[thm]{Proposition}
\newtheorem{xca}[thm]{Exercise}
\newtheorem{rem}{Remark}[section]
\newtheorem{cor}{Corollary}[section]
\newtheorem{method}[thm]{Method}
\def\bt{\begin{thm}}
\def\et{\end{thm}}

\numberwithin{equation}{section}
\numberwithin{figure}{section}

\def\la{\label}

\keywords{Cahn-Hilliard equation, pattern formation, center manifold reduction, dynamical transition theory, lattice structures, rolls, squares, hexagons}

\maketitle

\begin{abstract}
This article examines the dynamic phase transitions and pattern formations attributed to binary systems modeled by the Cahn-Hilliard equation. In particular, we consider a two-dimensional lattice structure and determine how different choices of the spanning vectors influence the resulting dynamical tramsitions and pattern formations. As the basic  steady-state loses its linear stability, the binary system undergoes a dynamic transition which is shown to be characterized by both the geometry of the domain and the choice of physical parameters of the model. Unlike rectangular domains, we are able to observe the emergence of hexagonally–packed circles, as well as the familiar rolls and square structures. We begin with the decomposition of our function space into a stable and unstable eigenspace before calculating the center manifold that maps the former to the later. In analyzing the resulting reduced equations, we consider the different multiplicities that the critical eigenvalue can have, which is shown to be geometry-dependent. We briefly consider the long-range interaction model and determine that it produces similar results to the original model. 
\end{abstract}

\tableofcontents

\section{Introduction}
\-\quad\ The Cahn-Hilliard model is a partial differential equation that describes the process of phase separation by which two components of a binary fluid spontaneously separate and form domains pure in each component; see, e.g., Cahn and Hilliard \cite{CH57}, Novick-Cohen and Segel \cite{NS84}, Reichl \cite{reichl}, and Pismen \cite{pmn}. There are intensive mathematical and physical studies on Cahn-Hilliard equations, including e.g. 
Ryan Goh and Arnd Scheel \cite{Scheel15} on pattern formation in the wake of triggered pushed fronts, Hansj\"org Kielh\"ofer \cite{Kielhofer} on stationary pattern solutions for the Cahn-Hilliard model, 
A. A. Nepomnyashchy \cite{Nepomnyashchy2015} on mechanisms that control coarsening and pattern formation, 
Nicolaenko-Scheurer-Temam \cite{Nicolaenko} on  inertial manifolds for pattern-forming CH dynamics, 
George R. Sell \cite{Sell} on  multi-peak/stationary solutions in CH models,  
S. Maier-Paape \& T. Wanner \cite{Wanner} on rigorous analysis of spinodal decomposition and wavelength selection in higher dimensions, 
R. Choksi \& P. Sternberg \cite{Choksi} on  CH energies with long-range interactions and periodic patterns, 
and  
A. Gusak \& K. Tu \cite{Gusak} on phase diagrams and spinodal decomposition and  segregation modeled by  Cahn-Hilliard.   The Cahn-Hilliard model is also used in modeling sharp interfaces of materials such as in Liu and Shen \cite{liu},  Shen and Yang \cite{sy} and the references therein,  for the mixture of two incompressible fluids and its approximations. Many situations can be modeled as a phase separation of binary systems, and the systematic study of solutions to the Cahn-Hilliard equation and their stabilities prove to be useful in the natural sciences. \\

The main objective of this paper is to initiate a study of dynamic transitions and pattern formations on a lattice periodic structure for the Cahn-Hilliard model without or with a long-range interaction. The specific goal is then to explore how the geometry of the spatial domain, the physical parameters $\gamma_2$ and $\gamma_3$, and the control $\lambda$ affect \textbf{1)} the type of phase transitions, \textbf{2)} the structure of the transition states, and \textbf{3)} the emergence of different patterns (rolls, squares, hexagons, etc.).  \\

This article will examine the phase transition and pattern formation that occurs in a lattice domain system. The control parameter $\lambda$ plays a critical role in determining the degeneracy of the basic solution $u=0$ into patterns in the form of new solutions to the model. Some patterns found in the lattice domain include rolls, squares, hexagons, and rectangles in the far field.

It is classical that the Cahn-Hilliard model can be put in the perspective of an infinite dimensional dissipative dynamical system. The mathematical analysis of the model is carried using the dynamical transition theory developed by Ma and Wang \cite{ptd}.  For many problems in sciences, we need to understand  the transitions  from one state to another, and the stability/robustness of the new states. For this purpose, a dynamic transition theory is developed Ma and Wang \cite{ptd}, and is applied to both equilibrium and non-equilibrium phase transitions in nonlinear sciences. The basic philosophy of the dynamic transition theory is to search for  the complete set of transition  states, which are represented by a  local attractor, rather than some steady states or periodic solutions or other type of orbits as part of this local attractor.  A starting point of the dynamic transition  theory  is the introduction of a dynamic classification scheme of dynamic transitions, with which  phase transitions, both equilibrium and non-equilibrium, are classified into three types: Type-I, Type-II and Type-III Mathematically, these transitions  are also respectively  called continuous, jump (or catastrophic) and mixed (or random) transitions. \\

Consider  the Cahn-Hilliard problem studied in this article. First, the key ingredients of the analysis consist of the following. First, the solution on a lattice structure $L$ with dual lattice $L^\ast =\{n_1k_1 + n_2k_2 | (n_1,n_2)\in \Z^2\}$ can be Fourier expanded, where $k_1$. and $k_2$ are spanning vectors for the dual lattice; see Section 2 for details. The Fourier modes correspond to eigenfunctions of the linearized equation. This leads to a precise characterization of the critical thresholds, the principle of exchange of stabilities (PES), and the stable and unstable modes. \\

Second, we derive leading order approximations of the center manifold function, so that the stable modes are written as functions of the unstable/center modes. We then derive the leading order approximation of the reduced system of the original Cahn-Hilliard on the center manifold. The reduced system depends on the number of critical modes, the spatial geometry, and the physical parameters $\gamma_2$, $\gamma_3$, and $\lambda$. The reduced system captures the precise information on types of phase transitions, the structure of the transitions, and the related emerging patterns. For example, in the case where the dimension of the critical space is four, the type of transition is dictated by the sign of \begin{align*}
b=\gamma_3-A\gamma_2^2,
\end{align*} 
with 
\begin{align}
    A=\frac{2}{9} \max \left[ \left(
    \frac{1}{3|k_1^c|^2}+ \frac{2}{|k_1^c-k_2^c|^2-|k_1^c|^2}+\frac{2}{|k_1^c+k_2^c|^2-|k_1^c|^2}\right), \frac{1}{|k_1^c|^2}\right].
\end{align}
Here $k_1^c=n_1^ck_1+n_2^ck_2$ and $k_2^c=n_3^ck_1+n_4^ck_2$ are critical vectors as defined e.g. in  (\ref{3.4}).
If $b>0$, the system undergoes a continuous dynamical transition to a local attractor $\Sigma_\lambda$, homological to a $3D$ sphere $S^3$. Also, $\Sigma_\lambda$ contains two circles of steady states and a two-dimensional torus of steady states. In addition, the solution on $\Sigma_\lambda$ gives rise to square and roll patterns. \\

Third, reduction of our model into a system of reduced equations allows us to find equilibrium solutions and transition types. As mentioned earlier, possible transitions include continuous, jump, and mixed  transitions. In a continuous transition, emerging solutions stay within a neighborhood of the basic solution during bifurcation. Conversely, a jump transition exhibits solutions quickly diverging from the basic solution and approaching the far field. Mixed transitions are those that exhibit behavior of both continuous and jump transitions. \\

Fourth, for call cases with continuous transitions studied, we obtained detailed structure of the set of all transition states $\Sigma_\lambda$, which is homological to an $m-1$ dimensional sphere $S^{m-1}$, consisting of steady states and their unstable manifolds. The steady states on $\Sigma_\lambda$ and their stability are explicitly derived as well. Different patterns of the transition states are examined. Also we like to mention that in the multiplicity six case, hexagon packed patterns also appear, which are absent in the classical Cahn-Hilliard model on rectangular domains.\\

The paper is organized as follows. The Cahn-Hilliard model and its mathematical section are given in Section 2. 
Linear instability and a general dynamical transition theorem are derived in Section 3. Section 4 states them main theorems characterizing the types of dynamical transitions. Theorem 4.1 on multiplicity 4 case is proved in Sections 5.1 \& 5.2. Section 5.3 studies the structure and patterns of the set of  transition states $\Sigma_\lambda \approx S^3$; see Theorem 5.1.  Section 6 studies the multiplicity 2 case and the main theorem is  Theorem~6.1.  Multiplicity 6 case is carefully investigated in Sections 7 \& 8. Section 8 deals with the case where the critical vectors satisfy $k_3^c = k_1^c + k_2^c$, and one has to deal carefully with zeroes appearing  in the denominators in some calculations.  Two appendices A \& B are given for a brief introduction to the dynamical transition theory  and the approximation of center manifold functions used in the article.\\

\noindent
{\bf Notations.} For convenience, here are some of the nations and terminologies used in the article.\\

- for physical parameters $b_1, b_2, b_3$,  see  (\ref{ch-1}) 

- for non-dimensional  parameters $\lambda, \gamma_2, \gamma_3$, see (\ref{2.5})

- for lattice $L$ and dual lattice $L^*$, see (\ref{L})   

- the dual lattice vector is given by $k=n_1k_1+n_2k_2$ with $k_1$ and $k_2$ being the generating vectors and $(n_1,n_2) \neq (0,0)$

- the critical vectors such as $
k_1^c=n_1^ck_1+n_2^ck_2$ are defined in (\ref{3.4}).

\section{Cahn-Hilliard Model}\label{s2}
Since binary systems are conserved, the equations describing the
Helmholtz process and the Gibbs process are the same. Hence,
without distinction we use the term "free energy" to discuss this
problem.  Let $u_A$   and $u_B$ be the concentrations of components $A$  and  $B$ respectively,  then $u_B=1-u_A$.  In a homogeneous state, $u_B={\bar u}_B$ is a constant.  We now consider the case where the Helmholtz  free energy is given by
\begin{equation}
F(u)=F_0+\int_{\Omega} \Big[\frac{\mu}{2}|\nabla u_B|^2+ f(u_B)\Big]dx, \label{8.49-2}
\end{equation}
and the density function $f$ is given by the Bragg-Williams potential, also called 
the Hildebrand theory \cite{reichl},  as follows:
\begin{align}
f= &  \mu_Au_A+\mu_Bu_B+RT[u_A \ln u_A + u_B\ln
u_B]+au_Bu_A,\label{8.122}
\end{align}
where $\mu_A,\mu_B$ are the chemical potentials of $A$ and $B$
respectively, $R$ the molar gas constant, $a>0$ the measure of
repulsion  action between $A$ and $B$.\\

In a homogeneous state, $u_B$ is a constant  in space and is equal to its mean value, i.e., $u_B= \bar u_B$. Let 
$$u=u_B-\bar{u}_B, \qquad u_0=\bar u_B.$$
Using the Landau mean field theory, we have the following Cahn-Hilliard equation
\begin{equation}\label{8.57}
\frac{\partial u}{\partial
t}=-\mu D\Delta^2u+\Delta [  b_1 u^1+ b_2 u^2+ b_3 u^3],
\end{equation}
where the coefficients $k$, $b_1$, $b_2$  and $b_3$  are given by 
\begin{equation}
\label{ch-1}
\begin{aligned}
& b_1 = \frac{D}{2} \frac{d^2 f(u_0)}{du^2}= \left[\frac{RT}{u_0(1-u_0)}- 2 a\right] \frac{D}{2},\\
& b_2=\frac{D}{3 !} \frac{d^3 f(u_0)}{du^3}= \frac{2u_0-1}{6u^2_0(1-u_0)^2}D RT,\\
& b_3=\frac{D}{4 !} \frac{d^4 f(u_0)}{du^4}= \frac{1-3u_0 +3u^2_0}{12u^3_0(1-u_0)^3}DRT, 
\end{aligned}
\end{equation}
where $D$ is the diffusion coefficient. \\

In this paper, we will explore the effort of the spatial geometry on   phase transitions and  pattern formations when  the spatial domain  $U$ is a lattice domain.  Let $l_1,l_2 \in \R^2$ be any set of linearly independent vectors. We consider a two-dimensional lattice $L$ and its dual lattice $L^\ast$ given by 
\begin{equation}\label{L}
    \begin{aligned}
       &  L=\{n_1l_1 + n_2l_2 \quad | \quad (n_1,n_2)\in \Z^2\},\\
        & L^\ast =\{n_1k_1 + n_2k_2\quad  | \quad (n_1,n_2)\in \Z^2\},
        \end{aligned}
\end{equation}
where $ k_i \cd l_j = 2\pi \delta_{ij} $ for $i,j \in \{1,2\}$. Let $U$ be the area enclosed by the parallelogram created by the vectors $l_1$ and $l_2$. \\

The non-dimensional form of  (\ref{8.57}) is given by 
\begin{equation}
\begin{aligned} &u_t=-\Delta ^2u-\lambda \Delta u +\Delta (\gamma_2u^2+\gamma_3u^3), \quad (x,t) \in \R^2 \times \R^+,\\
&u(x+l,t)=u(x,t),  \quad  \forall l\  \in L,\\
    &u(x,0)=\phi(x),\\
    &\int_{U}u(x,t)dx=0
.\end{aligned}\label{2.5}
\end{equation}
The non-dimensional variables and parameters with primes suppressed in (\ref{2.5})  are given, as in \cite{ptd}, by:
\begin{align*}
&x=lx^{\prime}, &&  t=\frac{l^4}{\mu D}t^{\prime}, && u=u_0u^{\prime},\\
&\lambda =-\frac{l^2b_1}{\mu D},&& \gamma_2=\frac{l^2b_2u_0}{\mu D},
&& \gamma_3=\frac{l^2b_3u^2_0}{\mu D},
\end{align*}
where $l$ is a given length,  $u_0=\bar{u}_B$ is the constant
concentration of $B$, and $\gamma_3>0$. \\

 The solutions of (\ref{2.5}) takes the form 
 \begin{align}u(x,t)=\sum\limits_{k \in L^* \backslash \{0\}}(z_{k}(t)e^{ik \cd x}+\conj{z_{k}(t)}e^{-ik \cd x}), \label{2.7}
 \end{align} 
where each point of the lattice can be written as $k=n_1k_1+n_2k_2$ for some integers $(n_1,n_2) \neq (0,0)$, as in Hoyle \cite{hoy}. 
Observe that the solution is periodic in $L$ as desired.
From this, it can be calculated that 
\begin{align}
& 
\label{2.8}
\Delta u(x,t)=\sum\limits_{k \in L^* \backslash \{0\}}-|n_1k_1 + n_2k_2|^2(z_k(t)e^{ik \cd x}+\conj{z_k(t)}e^{-ik \cd x}), \\
& \label{2.9} \Delta^2 u(x,t)=\sum\limits_{k \in L^* \backslash \{0\}}|n_1k_1 + n_2k_2|^4(z_k(t)e^{ik \cd x}+\conj{z_{k}(t)}e^{-ik \cd x}).
\end{align}
To put the model (\ref{2.5}) in the perspective of nonlinear dissipative dynamical systems, we let
\begin{equation}\begin{aligned}
& H  = \{u \in L^2(U) \mid \int_{U}udx=0\},\\
& H_1 = \{u \in H^4(U) \cap H \mid u(x+l,t)=u(x,t), l \in L\},\\
& H_{1/2} = \{u \in H^2(U) \cap H \mid u(x+l,t)=u(x,t), l \in L\}.\end{aligned}\end{equation}
We shall split the linear component of (\ref{2.5}) into two operators: $L_\lambda = -A+B_\lambda:H_1 \rightarrow H$ and write the nonliear term as $G:H_{1/2} \rightarrow H$ by
\begin{equation}
 Au = \Delta^2u, \quad B_\lambda u = -\lambda u, \quad G(u)= \Delta(\gamma_2u^2+\gamma_3u^3). 
 \label{2.11}\end{equation} 
 Then, (\ref{2.5}) can be written as 
 \begin{equation}
 \begin{aligned}
 & \frac{\partial u}{\partial t} = L_\lambda u + G(u),\\
 & u(x,0)= \phi (x).\end{aligned}\label{2.12}\end{equation} 
 It is then classical to show that for any $\phi \in H$, (\ref{2.12}) has a global in time solution \begin{align}
    u\in L^2([0,T];H_1)\cap L^\infty ([0,T];H), \text{ for all }T>0.
\end{align} In other words, (\ref{2.12}) is a well-posed dynamical system.

\section{Dynamic Transitions and Principle of Exchange of Stabilities}

To study the dynamical transitions and pattern formations of (\ref{2.12}), we first examine the linear instability. Consider the linear eigenvalue problem $L_\lambda w=\beta(\lambda) w$ given by 
\begin{align*}
& -\Delta ^2 w -\lambda \Delta w = \beta(\lambda) w,\\
& w(x+l)=w(x),  \quad  \forall l\  \in L, \\
& \int_{U}w(x)dx=0.
\end{align*}
 
 In view of (\ref{2.8}) and (\ref{2.9}), the eigenvalues and eigenfunctions of $L_\lambda$  are 
\begin{align}& \beta_{n_1n_2}(\lambda)=-|n_1k_1+n_2k_2|^4+\lambda |n_1k_1+n_2k_2|^2  \label{3.1}\\
& \qquad \qquad  \nonumber = -|n_1k_1+n_2k_2|^2 (|n_1k_1+n_2k_2|^2-\lambda ),\\
& e_{n_1n_2}=e^{i(n_1k_1 \cd x + n_2k_2 \cd x)}.
\end{align}
Note that with notation $k=n_1k_1+n_2k_2$, these can be written equivalently as 
$$\beta_k=-|k|^2(|k|^2-\lambda), \qquad e_k=e^{ik \cd x},$$ 
and will be used interchangeably henceforth. \\

Let $S\subset \Z^2\setminus \{(0,0)\}$ be the set of all integer weights $(n_1,n_2)$ that minimize the magnitude of the vector $k=n_1k_1+n_2k_2$. More explicitly, denote 
\begin{align}
S&=\{(n_1,n_2)\in \Z^2 \setminus \{(0,0)\}\text{ }|\\ 
\label{3.3}
& \qquad \quad  |n_1k_1+n_2k_2|^2 
= \min_{(p,q)\in \Z^2\setminus \{(0,0)\}}|pk_1+qk_2|^2\}. \nonumber 
\end{align}
It can be seen that the possible values of the cardinality of $S$ are two, four, and six. For notation, when $\# S=6$, the elements of $S$ are $(n_1^c,n_2^c)$, $(-n_1^c,-n_2^c)$, $(n_3^c,n_4^c)$, $(-n_3^c,-n_4^c)$, $(n_5^c,n_6^c)$, and $(-n_5^c,-n_6^c)$. Define the critical vectors of $L^*$ by 
\begin{equation}\label{3.4}
\begin{aligned}
k_1^c=n_1^ck_1+n_2^ck_2,\\k_2^c=n_3^ck_1+n_4^ck_2,\\k_3^c=n_5^ck_1+n_6^ck_2.
\end{aligned}
\end{equation} 
Should $\#S=2$ or $\#S=4$, we will work only with one or two critical vectors, respectively. \\

We now give examples of three lattices in which the critical eigenvalue has multiplicity two, four, and six. When $\#S=4$, the four vectors that have minimal magnitude are the two critical vectors and their opposites. Perhaps the simplest lattice to consider is the square lattice spanned by the vectors $k_1=(1,0)$ and $k_2=(0,1)$. In this case, the elements of the lattice that have minimal magnitude are the two spanning vectors as well as their opposites. \\

When $\#S=2$, only one vector and its additive inverse can achieve minimal magnitude, as is the case with the critical vector $k^c=(\frac{\sqrt{3}}{2}-1,\frac{1}{2})$ of the lattice spanned by the vectors $k_1=(1,0)$ and $k_2=(\frac{\sqrt{3}}{2}, \frac{1}{2})$. Notice that this critical vector is the difference of the two spanning vectors; that is $k^c=k_2-k_1$. \\

When $\#S=6$, we seek three vectors and their opposites with minimal magnitude. The lattice spanned by the vectors $k_1=(1,0)$ and $k_2=(-\frac{1}{2},-\frac{\sqrt{3}}{2})$ has critical vectors $k_1$, $k_2$, and $k_1+k_2=(\frac{1}{2},-\frac{\sqrt{3}}{2})$ as well as their additive inverses. Figure~\ref{f3.1} shows graphs of these three lattices.\\

\begin{figure}
    \centering
    \includegraphics[width=0.99\linewidth]{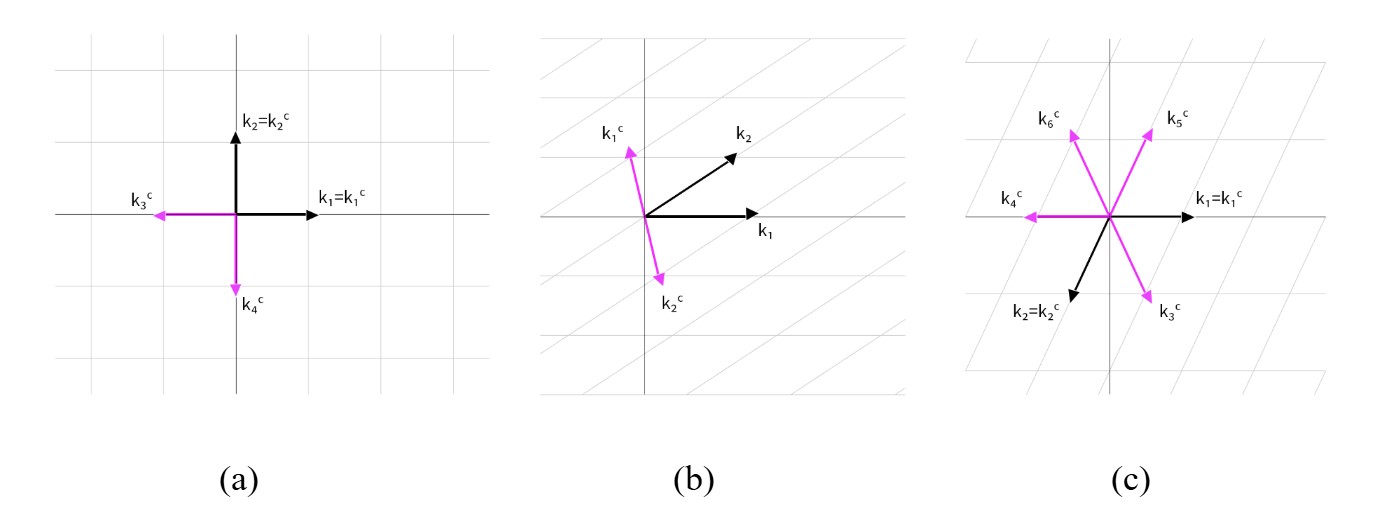}
    \caption{Examples of various lattice structures in which the critical eigenvalue has multiplicity; \textbf{(a)} four, \textbf{(b)} two, and \textbf{(c)} six.}
    \label{f3.1}
\end{figure}

Define the spaces $E_1^\lambda$ and $E_2^\lambda$ by \begin{equation}
    \begin{aligned}
        E_1^\lambda &= \text{span} \{e^{\pm i (n_1^c k_1 + n_2^c k_2) \cd x} \text{ } | \text{ } (n_1^c,n_2^c)\in S\},\\
        E_2^\lambda &= \overline{\text{span} \{u\text{ }|\text{ }\langle u,e_i\rangle = 0 \text{ for all } e_i \in E_1^\lambda \}},
    \end{aligned}
\end{equation}
where the closure is taken in $H$-norm.\\

We define the critical value of the control parameter by 
\begin{equation} \label{critical}
\lambda_0 = |k_1^c|^2 = |k_2^c|^2=|k_3^c|^2.
\end{equation}
 This critical value plays a central role in the stability of the basic solution $u=0$. Specifically, as $\lambda$ crosses the threshold $\lambda_0$, a finite number of the eigenvalues given by (3.1) become positive, and the basic solution becomes linearly unstable. 
This principle of exchange of stabilities (PES) is given mathematically as follows:
\begin{equation}
\begin{aligned}
&
\beta_{n}(\lambda)
 \begin{cases} 
  <0  & \text{if $\lambda<\lambda_0$}, \\
  =0 & \text{if $\lambda=\lambda_0$},\\
 >0 & \text{if $\lambda>\lambda_0$},\\
 \end{cases}  && \text{ for } n \in S,
 \\
 &\beta_{n_1n_2}(\lambda_0)<0 && \text{ for } (n_1,n_2)\in \Z \times \Z \setminus (\{(0,0)\}\cup S).
\end{aligned}
\end{equation}
 Thus, with critical value $\lambda_0$, the eigenvalue $-|k_1^c|^2 (|k_1^c|^2 -\lambda)$, has either multiplicity two, four, or six with basis vectors in the multiplicity six case for the eigenspace 
 \begin{equation}\begin{aligned} 
 & \{e_1 = e^{i(k_1^c \cd x)},e_2 = e^{-i(k_1^c \cd x)},e_3 = e^{i(k_2^c\cd x)},\\
 & \qquad e_4 = e^{-i(k_2^c \cd x)},e_5=e^{i(k_3^c\cdot x)}, e_6=e^{-i(k_3^c\cdot x)}\} ,\end{aligned}\end{equation} with the number of vectors equal to the cardinality of $S$. We now verify the existence of a dynamical phase transition of (\ref{2.12}) as $\lambda$ becomes larger than $\lambda_0$. Based on \cite[Theorem 2.1.3]{ptd}, which is restated in Theorem~\ref{tA.1} in the appendix, we have the following dynamical transition theorem:

\begin{theorem}[Existence of Transition]\label{tm3.1}
    The system (\ref{2.12}) undergoes a dynamical transition from the basic state $u=0$ as the control parameter $\lambda$ crosses the critical threshold $\lambda_0$. The transition is one of the three types: continuous, catastrophic, or random. 
\end{theorem}

\begin{remark}
  {\rm  The type of transition in this theorem is dictated by the nonlinear interaction of stable and unstable modes, and  is dependent on the geometry of the lattice and the system parameters $\lambda$, $\gamma_2$, and $\gamma_3$. The theory and technical tools of the analysis is dynamical transition theory and center manifold techniques. This will be the main focus of the remaining part of the paper.
}
\end{remark}

\section{Main Dynamic Transition Theorems}

\subsection{Dynamical transition theorems}
As we stated in Theorem~\ref{tm3.1},  the system (\ref{2.12}) always undergoes a dynamical transition from the basic state $u=0$ as the control parameter $\lambda$ crosses the critical threshold $\lambda_0$. We are now ready to determine the type of transitions, dictated by the nonlinear interaction of stable and unstable modes.

First we consider the case where $S$ has cardinality four. Let 
\begin{equation}
\begin{aligned}E_1^\lambda&=\text{span}\{e_1(\lambda), e_2(\lambda), e_3(\lambda), e_4(\lambda)\},  \end{aligned}
\end{equation}
denote the unstable and stable eigenspaces, respectively.  Here 
$$
 \{e_1 = e^{i(k_1^c \cd x)},\quad e_2 = e^{-i(k_1^c \cd x)}, \quad e_3 = e^{i(k_2^c\cd x)},  \quad e_4 = e^{-i(k_2^c \cd x)}\}. $$

We have the following main dynamic transition  theorem.
\begin{theorem}[Transition Types of Multiplicity Four Case]\label{tm4}
Consider system (\ref{2.12}). Let the multiplicity of $\beta_1(\lambda_0)$ be four. Let 
\begin{align}\label{4.2}
    A=\frac{2}{9} \max \left[ \left(
    \frac{1}{3|k_1^c|^2}+ \frac{2}{|k_1^c-k_2^c|^2-|k_1^c|^2}+\frac{2}{|k_1^c+k_2^c|^2-|k_1^c|^2}\right), \frac{1}{|k_1^c|^2}\right].
\end{align}
Then the following assertions hold true: 
\begin{enumerate}

\item If \begin{align}\gamma_3 > A\gamma_2 ^2,\end{align} then as $\lambda$ crosses $\lambda_0$, the system (2.11) undergoes a continuous (Type I) dynamic transition to a local attractor $\Sigma_\lambda$, homological to $S^3$.\\
\item If \begin{align}\gamma_3 < A\gamma_2 ^2,
\end{align} 
then as $\lambda$ crosses $\lambda_0$, the system undergoes a catastrophic  (Type II) dynamic transition.
\end{enumerate}
\end{theorem} 

A few remarks are now in order. 

\medskip

{\bf 1).}   If the system undergoes a Type II transition, then the following are true:
 \begin{enumerate}
\item Let $\lambda <\lambda_0$ and $\lambda$ be near $\lambda_0$, then the system undergoes a sub-critical bifurcation.\\
   \item There exists $\lambda^\ast<\lambda_0$ at which a saddle node bifurcation occurs.
   \end{enumerate}
    These statements are true regarding all Type II transitions in future theorems.

\medskip

{\bf 2).}  In the case where  the pairs of $(n_1,n_2)$ that minimize $|n_1k_1+n_2k_2|^2$ are $(n_1^c,n_2^c)$ and $(-n_1^c,-n_2^c)$, we can derive a similar theorem, Theorem~\ref{tm2},  on the dynamic transition of the system. In this case, the parameter $A$ given in (\ref{4.2}) takes a much simpler form as 
\begin{equation}
\label{4.5}
A= \frac{2}{9|k_c|^2}.
\end{equation}

\medskip

{\bf 3).} The multiplicity six case is more delicate, and is studied in Section~\ref{s7}  for the non-resonant case, and in Section \ref{s8} for  the resonant case. For the multiplicity six case, the critical wave modes $k_1^c$, $k_2^c$. and $k_3^c$ are given by (\ref{3.4}); see also (\ref{7.1}).

The non-resonant case is defined as the case where  $k_3^c = a k_1^c + b k_2^c$ such that $a,b \neq 0$ and $(a,b) \in \Z^2 \setminus \{(1,1),(-1,1),(-1,-1),(1,-1)\}$. Therefore in  the resonant case $k_3^c$ is a linear combination of $k_1^c$ and $k_2^c$ with $\pm1$-coefficients, such as 
$$k_3^c = k_1^c + k_2^c.$$

For the non-resonant case, the parameter $A$ becomes much more complex, and is given explicitly by (\ref{7.8}) and (\ref{7.9}). The dynamic transition theorem in this case is then stated  in Theorem~\ref{tm7.1}.\\

Notice that in all above cases, the lowest nonlinear terms in  the center-manifold reduced equations are cubic, leading either continuous or catastrophic dynamic transitions.  In the resonance case, the reduced systems contains quadratic terms; see (\ref{8.3}) for the general resonant case, and (\ref{8.10}) for the case with evenness condition in spatial variable $x$. The analysis of (\ref{8.10}) is a bit involved and will be studied elsewhere. With the  evenness assumption, the dynamic transition is fully characterized in Theorem~\ref{tm8.1}.

\medskip

{\bf 4).} For the Cahn-Hilliard equation with long-range interaction, more complex patterns can emerge as the critical wave vector $k^c=n_1^ck_1+n_2^ck_2$ is related to the long-range interaction term in the manner given by  (\ref{9.5}):
 $$
 |k^c|^2 \sim \sqrt{\sigma}.
$$
 Consequently, the larger the $\sigma$, the larger $|k_c|$, leading to the formation of more complex patterns. \\
 
 In addition, in this case the parameter $A$ dictating the transition is given by 
$$A=\frac{8|k_1^c|^2}{9(4|k_1^c|^4 - \sigma)}.$$

\subsection{Pattern formation and structure of transition states}
In the theorems discussed above, when the transition is continuous, the set of transition states, $\Sigma_\lambda$ is homological to an $m-1$ dimensional sphere $S^{m-1}$, where $m$ is the dimension of unstable space. Another objective of the paper is to characterize the detailed 
structure of $\Sigma_\lambda$ and patterns emerged in the transition states.\\

The structure $\Sigma_\lambda \sim S^3$ given in Theorem~\ref{tm4} is studied in Section~\ref{s5.3} with examples given in Section~\ref{s5.4}.
In particular, we show that  $\Sigma_\lambda$ contains two circles of steady states and one tori of steady states.\\

Other cases are analyzed in detail in Sections~\ref{s7.2}, \ref{s7.3}, \ref{s8.3} and \ref{s8.4}.

\section{Proof of Theorem~\ref{tm4}}
The main idea of the proof  is 1)  to reduce the original PDE system to the center/unstable manifold generated by the unstable modes, as $\lambda$ crosses the critical value $\lambda_0$, and 2) to analyze the reduced system on the center manifold. 

\subsection{Center manifold reduction}
  Let $P_i$ denote the canonical projection of $H$ into $E_i$ for $i=1,2$. Then \textit{u(x,t) }belongs to the direct sum of these two spaces and can be written as 
  \begin{align} u(x,t)=y_1e_1+y_2e_2+y_3e_3+y_4e_4+z,\label{new4.2}
  \end{align} 
  where \(z\in E_2^\lambda\) is the stable component. Equation (\ref{2.12}) can thus be written in the form \begin{align}(y_t\cdot e+ z_t)=L_\lambda (y\cdot e +z)+G(y\cdot e+z),\label{4.3}
  \end{align} 
  where \(y=(y_1,y_2,y_3,y_4)\) and \(e=(e_1,e_2,e_3,e_4)\). 
Let \(e_i\in e\) be any of the four unstable eigenfunctions. Consider the projection of the Cahn-Hilliard equation that results from taking the inner product of both sides of (\ref{4.3}) with \(e_i\). Using the mutual orthogonality of the eigenfunctions, we have \begin{align}\int_{U} \dot{y_i} e_i\bar{e_i} dx=\int_{U} L_\lambda (y_i e_i) \bar{e_i}dx + \int_{U} G(y\cdot e+z) \bar{e_i}dx.\end{align} Here integration is to be taken over the parallelogram \(U\). Applying the linear operator to \(e_i\) and simplifying, we are left with \begin{align}\dot{y_i} \int_{U} e_i \bar{e_i}dx=y_i \beta_{n_1^cn_2^c}(\lambda) \int_{U} e_i \bar{e_i}dx + \int_{U} G(y \cdot e + z) \bar{e_i}dx,\nonumber \end{align} which implies that
\begin{align}\dot{y_i}=y_i \beta_{n_1^cn_2^c}(\lambda)+ \frac{1}{|e_i|^2} \int_{U} G(y \cdot e + \Phi (y)) \bar{e_i}dx,\label{4.6}\end{align}
where the norm of the eigenfunction is the area of the parallelogram \(|l_1 \times l_2|\). Here we have written \(z=\Phi(y)\) which is the center manifold function that maps the unstable eigenspace to the stable. 
Because the linear operator \(L\) is diagonal near \(\lambda = \lambda_0\), we can use the following formula for the center manifold; see Theorem~\ref{tB.1} in the appendix, which is \cite[Theorem A.1.1]{ptd}) recalled here: 
\begin{align}-L_\lambda \Phi(y)=P_2G(y\cdot e) +h.o.t.,\label{4.7}
\end{align} 
where \(P_2\) denotes projection onto the stable eigenspace \(E_2^\lambda\). Fourier decomposing the center manifold as \(\Phi(y)=\sum_{k=5}^{\infty} \varphi_k (y)e_k\) allows us to write the left hand side of (\ref{4.7}) as \begin{align}-L_\lambda \Phi (y)=-\sum_{k=5}^{\infty}\varphi_k L_\lambda e_k=-\sum_{k=5}^{\infty} \varphi_k \beta_k (\lambda) e_k.\end{align} Let \(e_k \in E_2^\lambda\) be any stable eigenfunction. Then taking the inner product of (\ref{4.7}) with \(e_k\) yields \begin{align}-\beta_k |e_k|^2 \varphi_k=\int_{U} G(y\cdot e) \bar{e_k}dx +  h.o.t.,\end{align} which implies that
\begin{align}\varphi_k(y)=-\frac{1}{\beta_k |e_k|^2}\int_{U} G(y\cdot e) \bar{e_k}dx +h.o.t.,\label{4.10}\end{align}
where \(\varphi_k\) represents the Fourier coefficient of \(\Phi(y)\) with basis function \(e_k\). We shall use the following eight stable eigenfunctions for our approximation of the center manifold: 
\begin{equation}\begin{aligned}e_{5,6}&=e^{\pm2ik_1^c \cdot x},&e_{7,8}&=e^{\pm2ik_2^c \cdot x},\\e_{9,10}&=e^{\pm i(k_1^c \cdot x + k_2^c \cdot x)},&e_{11}&=e^{i(-k_1^c \cdot x + k_2^c \cdot x)},\\e_{12}&=e^{i(k_1^c \cdot x - k_2^c \cdot x)}.\end{aligned}\end{equation} Associated with each of these eigenfunctions is a coefficient in \(y\) that is found by means of equation (\ref{4.10}). Let \(e_k \in \{e_i\}_{i=5}^{i=12}\) be any one of these stable eigenfunctions. Per equation (\ref{4.10}), we have \begin{align}\varphi_k(y)=-\frac{1}{\beta_k |e_k|^2}\int_{U} \Delta(\gamma_2 (y\cdot e)^2+\gamma_3(y\cdot e)^3)\bar{e_k}dx.\end{align}
Using integration by parts and the fact that \(\Delta \bar{e_k}=|k|^2\bar{e_k}\), we can write
\begin{align}\varphi_k(y)=-\frac{|k|^2}{\beta_k |e_k|^2} \int_{U} (\gamma_2 (y\cdot e)^2+\gamma_3(y\cdot e)^3)\bar{e_k}dx.\label{4.13}\end{align}
The orthogonality of the eigenfunctions ensures that \(\int e_i\bar{e_j}dx = |e_j|^2\delta_{ij}\). When the nonlinear term of (\ref{4.13}) is expanded,  we see that the integral of every term vanishes save for those of \(\gamma_2(y \cdot e)^2+\gamma_3(y\cdot e)^3\) that have eigenfunction \(e_k\). More specifically, take \(e_5=e^{2ik_1^c\cdot x}\). Dropping the cubic term, we have \begin{align} \varphi_5(y) = -\frac{|k_{2n_1^c2n_2^c}|^2}{\beta_{2n_1^c2n_2^c}|e_5|^2}\int_U \gamma_2(y_1e_1+y_2e_2+y_3e_3+y_4e_4)^2\bar{e_5}dx,\end{align}
where the subscript on \(k_{2n_1^c2n_2^c}\) denotes the integers in the linear combination of the eigenfunction \(e_5=e^{i(2k_1^c\cdot x + 0k_2^c\cdot x)}\). Since \(\bar{e_5}=e^{-2ik_1^c\cdot x}\), we are searching for terms of the square that contain the eigenfunction \(e^{2ik_1^c\cdot x}\). Upon expanding, we observe that the only such term is \(\gamma_2y_1^2e_1^2\), since \(e_1^2=e_5\). Therefore the integral of all other terms is zero and we are left with
\begin{align}
\varphi_5(y)&=-\frac{|k_{2n_1^c2n_2^c}|^2}{\beta_{2n_1^c2n_2^c}|e_5|^2}\int_{U} \gamma_2y_1^2e_1^2\bar{e_5}dx  \\
&=-\frac{\gamma_2|k_{2n_1^c2n_2^c}|^2}{\beta_{2n_1^c2n_2^c}|e_5|^2} y_1^2 \int_{U} dx\nonumber \\
&=-\frac{\gamma_2|k_{2n_1^c2n_2^c}|^2}{\beta_{2n_1^c2n_2^c}}y_1^2. \nonumber\end{align}
where again \(k_{2n_1^c2n_2^c}=2k_1^c+2k_2^c\) and \(\beta_{2n_1^c2n_2^c}=-|k_{2n_1^c2n_2^c}|^2(|k_{2n_1^c2n_2^c}|^2-\lambda)\) is the eigenvalue of \(e_5\). 
We can obtain the other Fourier coefficients of the center manifold by proceeding in a similar manner for the rest of the stable eigenfunctions. Listed here, the coefficients are

\begin{equation}
\begin{aligned}\varphi_{5,6,7,8}&=-\frac{\gamma_2 |k_{(2n_1^c2n_2^c)}|^2}{\beta_{(2n_1^c2n_2^c)}}y_{1,2,3,4}^2,\\\varphi_9&=-\frac{2\gamma_2 |k_{(n_1^c+n_3^c,n_2^c+n_4^c)}|^2}{\beta_{(n_1^c+n_3^c,n_2^c+n_4^c)}} y_1y_3,\\\varphi_{10}&=-\frac{2\gamma_2 |k_{(-n_1^c-n_3^c,-n_2^c-n_4^c)}|^2}{\beta_{(-n_1^c-n_3^c,-n_2^c-n_4^c)}} y_2y_4,\\\varphi_{11}&=-\frac{2\gamma_2 |k_{(-n_1^c+n_3^c,-n_2^c+n_4^c)}|^2}{\beta_{(-n_1^c+n_3^c,-n_2^c+n_4^c)}} y_2y_3,\\\varphi_{12}&=-\frac{2\gamma_2 |k_{(n_1^c-n_3^c,n_2^c-n_4^c)}|^2}{\beta_{(n_1^c-n_3^c,n_2^c-n_4^c)}} y_1y_4.\label{4.16}\end{aligned}\end{equation}
Writing the center manifold as \(\Phi(y)=\sum_{k=5}^{12} \varphi_k(y) e_k\) and returning to equation (\ref{4.6}), we can view the nonlinear term as
\begin{align}\gamma_2 P_1 \Delta (\sum_{j=1}^{4}y_j e_j+ \Phi)^2+\gamma_3 P_1\Delta (\sum_{j=1}^{4}y_je_j + \Phi)^3.\end{align}
Letting \(e_k \in E_1^\lambda\) and dropping higher-order terms, we have
\begin{align}\gamma_2 \int_{U} \Delta (\sum_{j=1}^{4}y_je_j)^2 \bar{e_k}dx +2\gamma_2 \int_{U} \Delta ((\sum_{j=1}^{4}y_je_j)\Phi) \bar{e_k} dx + \gamma_3 \int_{U} \Delta (\sum_{j=1}^{4}y_je_j)^3 \bar{e_k} dx.\end{align}
Using integration by parts and the fact that \(\Delta \bar{e_k}= |k|^2 \bar{e_k}\), we're left with
\begin{align} |k|^2 (\gamma_2 \int_{U} (\sum_{j=1}^{4}y_je_j)^2\bar{e_k}dx + 2\gamma_2 \int_{U} (\sum_{j=1}^{4}y_je_j) \Phi \bar{e_k}dx + \gamma_3 \int_{U} (\sum_{j=1}^{4}y_je_j)^3\bar{e_k}dx).\label{4.19}\end{align}
Consider the differential equation for \(y_1\) in the system (\ref{4.6}). We have, along with (\ref{4.19}), that
\begin{align}
\dot{y_1}&=y_1\beta_{n_1^cn_2^c}(\lambda)+\frac{|k_1^c|^2}{|e_1|^2}(\gamma_2\int_{U}(\sum_{j=1}^{4}y_je_j)^2\bar{e_1}dx \label{4.20} \\
&+2\gamma_2\int_{U}(\sum_{j=1}^{4}y_je_j)\Phi \bar{e_1}dx+\gamma_3\int_U(\sum_{j=1}^{4}y_je_j)^3\bar{e_1}dx).
\nonumber \end{align}
We search for cross-terms among \((\sum_{j=1}^{4}y_je_j)^2\) that contain the eigenfunction \(e_1=e^{ik_1^c\cdot x}\). However, observation shows that there are no such terms. This holds true, in fact, for each of the four differential equations in (\ref{4.6}). Now looking at the term \((\sum_{j=1}^{4}y_je_j)\Phi\) and using our expansion of the center manifold, we see that the terms
\begin{equation}\begin{aligned} y_2e_2\varphi_5e_{5}, \text{ } y_3e_3\varphi_{12}e_{12}, \text{ } y_4e_4\varphi_9e_9 ,\nonumber \end{aligned}\end{equation}
all contain the eigenfunction \(e_1\) once simplified. Thus, the only non-zero terms that result from the second integral in (\ref{4.20}) are 
\begin{align} 2\gamma_2|e_1|^2(y_2\varphi_5+y_3\varphi_{12}+y_4\varphi_9). \nonumber \end{align} 
Finally, consider the term \((\sum_{j=1}^{4}y_je_j)^3\). We again seek cross-terms of this cube that contain the eigenfunction \(e_1\) so that the integral may be non-trivial. The terms we find are
\begin{align} 3y_1^2e_1^2y_2e_2 , \text{ } 6y_1e_1y_3e_3y_4e_4. \nonumber \end{align} 
Thus the final integral in (\ref{4.20}) becomes 
\begin{align} \gamma_3|e_1|^2(y_1^2y_2+y_1y_3y_4). \nonumber \end{align}
Combining these results in the original differential equation (\ref{4.6}) yields the following reduced equation
\begin{align} \dot{y_1} = y_1\beta_{n_1^cn_2^c}(\lambda) + |k_1^c|^2(2\gamma_2(y_2\varphi_5+y_3\varphi_{12}+y_4\varphi_9)+\gamma_3(3y_1^2y_2+6y_1y_3y_4)). \label{4.26}\end{align}
Using the coefficients found in (\ref{4.16}) and the expressions of the eigenvalues in (\ref{3.1}), we can write \(y_2\varphi_5+y_3\varphi_{12}+y_4\varphi_9\) as
\begin{equation}\begin{aligned} \gamma_2(-\frac{y_1^2y_2}{|k_{(2n_1^c,2n_2^c)}|^2-\lambda}-\frac{2y_1y_3y_4}{|k_{(n_1^c-n_3^c,n_2^c-n_4^c)}|^2-\lambda}-\frac{2y_1y_3y_4}{|k_{(n_1^c+n_3^c,n_2^c+n_4^c)}|^2-\lambda})\\=\gamma_2(-\frac{y_1^2y_2}{4|k_1^c|^2-\lambda}-\frac{2y_1y_3y_4}{|k_1^c-k_2^c|^2-\lambda}-\frac{2y_1y_3y_4}{|k_1^c+k_2^c|^2-\lambda}).\nonumber \end{aligned}\end{equation}
Using this, we can write (\ref{4.26}) more conveniently as 
\begin{align*}
 \dot{y_1}&=-|k_1^c|^2(|k_1^c|^2-\lambda)y_1 \\
 &-2\gamma_2^2|k_1^c|^2(\frac{-2y_1y_3y_4}{|k_1^c-k_2^c|^2-\lambda}-\frac{2y_1y_3y_4}{|k_1^c+k_2^c|^2-\lambda}+\frac{y_1^2y_2}{-4|k_1^c|^2-\lambda})\nonumber \\
 &+\gamma_3|k_1^c|^2(3y_1^2y_2+6y_1y_3y_4).
 \end{align*}
Following this procedure for the other three differential equations in (\ref{4.6}) yields a system of reduced equations which will be used to determine the transition type and stability of solutions of the Cahn-Hilliard equation. \\

By omitting the higher order terms $o(3) \vcentcolon = o(|y^3|) + O(|y^3|\beta_1(\lambda)|) $, the reduced system of (2.11) on the center manifold is given by the following: 
\begin{align*}
    y_{1t}=&  \beta_1(\lambda)y_1 - 2\gamma_2^2 |k_1^c|^2 (\frac{-2y_1y_3y_4}{|k_1^c-k_2^c|^2-\lambda}-\frac{2y_1y_3y_4}{|k_1^c+k_2^c|^2-\lambda}+\frac{y_1^2y_2}{-4|k_1^c|^2+\lambda})\\
   & -\gamma_3 |k_1^c|^2(3y_1^2y_2+6y_1y_3y_4),\\
    y_{2t}=& \beta_1(\lambda)y_2 - 2\gamma_2^2 |k_1^c|^2 (\frac{-2y_2y_3y_4}{|k_1^c-k_2^c|^2-\lambda}-\frac{2y_2y_3y_4}{|k_1^c+k_2^c|^2-\lambda}+\frac{y_1y_2^2}{-4|k_1^c|^2+\lambda})\\
   & -\gamma_3 |k_1^c|^2(3y_1y_2^2+6y_2y_3y_4),\\
    y_{3t}=. &.  \beta_1(\lambda)y_3 - 2\gamma_2^2 |k_1^c|^2 (\frac{-2y_1y_2y_3}{|k_1^c-k_2^c|^2-\lambda}-\frac{2y_1y_2y_3}{|k_1^c+k_2^c|^2-\lambda}+\frac{y_3^2y_4}{-4|k_1^c|^2+\lambda})\\
  &  -\gamma_3 |k_1^c|^2(3y_3^2y_4+6y_1y_2y_3),\\
    y_{4t}=&. \beta_1(\lambda)y_4 - 2\gamma_2^2 |k_1^c|^2 (\frac{-2y_1y_2y_4}{|k_1^c-k_2^c|^2-\lambda}-\frac{2y_1y_2y_4}{|k_1^c+k_2^c|^2-\lambda}+\frac{y_3y_4^2}{-4|k_1^c|^2+\lambda})\\
   & -\gamma_3 |k_1^c|^2(3y_3y_4^2+6y_1y_2y_4).
\end{align*}
Letting $\lambda = \lambda_0 = |k_1|^2$, $y_2 = \conj{y_1}$, and $y_4 = \conj{y_3}$ gives \begin{align}
    y_{1t}&=  2\gamma_2^2 |k_1^c|^2 (\frac{2y_1|y_3|^2}{|k_1^c-k_2^c|^2-|k_1^c|^2}+\frac{2y_1|y_3|^2}{|k_1^c+k_2^c|^2-|k_1^c|^2}+\frac{y_1|y_1|^2}{3|k_1^c|^2})\nonumber \\&-\gamma_3 |k_1^c|^2(3y_1|y_1|^2+6y_1|y_3|^2),\\
    \conj{y_{1t}}&=  2\gamma_2^2 |k_1^c|^2 (\frac{2\conj{y_1}|y_3|^2}{|k_1^c-k_2^c|^2-|k_1^c|^2}+\frac{2\conj{y_1}|y_3|^2}{|k_1^c+k_2^c|^2-|k_1^c|^2}+\frac{\conj{y_1}|y_1|^2}{3|k_1^c|^2})\nonumber \\&-\gamma_3 |k_1^c|^2(3\conj{y_1}|y_1|^2+6\conj{y_1}|y_3|^2),\\
    y_{3t}&=  2\gamma_2^2 |k_1^c|^2 (\frac{2y_3|y_1|^2}{|k_1^c-k_2^c|^2-|k_1^c|^2}+\frac{2y_3|y_1|^2}{|k_1^c+k_2^c|^2-|k_1^c|^2}+\frac{y_3|y_3|^2}{3|k_1^c|^2})\nonumber \\&-\gamma_3 |k_1^c|^2(3y_3|y_3|^2+6y_3|y_1|^2),\\    
    \conj{y_{3t}}&=  2\gamma_2^2 |k_1^c|^2 (\frac{2\conj{y_3}|y_1|^2}{|k_1^c-k_2^c|^2-|k_1^c|^2}+\frac{2\conj{y_3}|y_1|^2}{|k_1^c+k_2^c|^2-|k_1|^2}+\frac{\conj{y_3}|y_3|^2}{3|k_1^c|^2})\nonumber \\&-\gamma_3 |k_1^c|^2(3\conj{y_3}|y_3|^2+6\conj{y_3}|y_1|^2).
\end{align}\\

Now break up the variables into real and imaginary components, so that $y_1 = a_1 + a_2i$ and $y_3 = a_3+a_4 i $. By combining real and imaginary parts, this gives the leading order approximation of reduced system at the critical threshold $\lambda_0$: \begin{equation}\begin{aligned}
    a_{1t}=a_1 (\xi (a_1^2+a_2^2) + \eta (a_3^2+a_4^2)),\\
    a_{2t}=a_2 (\xi (a_1^2+a_2^2) + \eta (a_3^2+a_4^2)),\\
    a_{3t}=a_3 (\eta (a_1^2+a_2^2) + \xi (a_3^2+a_4^2)),\\
    a_{4t}=a_4 (\eta (a_1^2+a_2^2) + \xi (a_3^2+a_4^2)),
\end{aligned}\end{equation} where 
\begin{equation}\begin{aligned}\xi &= \frac{2\gamma _2^2 -9|k_1^c|^2 \gamma _3}{3},\\\eta &= 2|k_1^c|^2 (\frac{2 \gamma_2^2}{|k_1^c-k_2^c|^2-|k_1^c|^2}+\frac{2\gamma_2^2}{|k_1^c+k_2^c|^2-|k_1^c|^2}-3\gamma_3).\label{4.37}\end{aligned}
\end{equation} 
It can be seen that for $i,j,k,l \in \{1,2,3,4\}$ where $i\neq j\neq k\neq l$, the 48 straight line orbits come from the 24 straight lines given by \begin{align}
    a_i&=a_j=a_k=0,\\
    a_1^2&=a_2^2=a_3^2=a_4^2,\\
    a_i&=a_j=0 \text{ and } a_k^2=a_l^2.
\end{align}\\

Finally, the reduced system \begin{equation}\begin{aligned}
    a_{1t}=\beta a_1 +a_1 (\xi (a_1^2+a_2^2) + \eta (a_3^2+a_4^2)) +o(3),\\
    a_{2t}=\beta a_2 +a_2 (\xi (a_1^2+a_2^2) + \eta (a_3^2+a_4^2)) +o(3),\\
    a_{3t}=\beta a_3 +a_3 (\eta (a_1^2+a_2^2) + \xi (a_3^2+a_4^2))+o(3),\\
    a_{4t}=\beta a_4 +a_4 (\eta (a_1^2+a_2^2) + \xi (a_3^2+a_4^2))+o(3),
\end{aligned}\end{equation} where $\beta = \beta_{10}(\lambda)$ and $o(3)=o(|a|^3)+O(|a|^3*|\beta|)$. By combining equations as well as letting $a_1^2+a_2^2=r_1^2$ and $a_3^2+a_4^2=r_2^2$ gives the following system  
\begin{equation}\begin{aligned}
    r_{1t}=\beta r_1 +r_1 (\xi r_1^2 + \eta r_2^2)+o(3),\\
    r_{2t}=\beta r_2 +r_2 (\eta r_1^2 + \xi r_2^2)+o(3).
\end{aligned}\end{equation}

\subsection{Analysis of the reduced system}
The transition type at the critical point $\lambda_0 = |k_1|^2$ is given by the system 
\begin{equation}\begin{aligned}
    r_{1t}=r_1 (\xi r_1^2 + \eta r_2^2), \\
    r_{2t}=r_2 (\eta r_1^2 + \xi r_2^2),
\end{aligned}
\end{equation} 
where $\xi $ and $\eta$ are defined in (\ref{4.37}). It can be calculated that 
\begin{align*}
    \xi + \eta >0 &\iff \gamma _3 < \frac29\left[\frac{1}{3|k_1^c|^2}+\frac{2}{|k_1^c-k_2^c|^2-|k_1^c|^2}+\frac{2}{|k_1^c+k_2^c|^2-|k_1|^2} \right]\gamma_2 ^2,\\
    \xi > 0 &\iff \gamma_3 < \frac{2}{9|k_1^c|^2}\gamma_2^2,\\
      \eta > 0 &\iff \gamma_3 < \frac{4|k_1^c|^2}{3}\left[\frac{1}{|k_1^c-k_2^c|^2-|k_1^c|^2}+\frac{1}{|k_1^c+k_2^c|^2-|k_1^c|^2}\right]\gamma_2^2.
\end{align*} 

Note that there are no elliptic orbits as the system (\ref{2.12}) is a gradient system (see Lemma A.2.7 in \cite{ptd}). Observe that on the straight line $r_1 = r_2$, the system satisfies $\frac{d r_2}{d r_1} = \frac{r_2}{r_1}$ where $\xi + \eta \neq 0$ and $r_1,r_2 >0$. It can be seen that $r_1 = 0$ and $r_2 = 0$ as well as $r_1 = r_2$ are the straight lines corresponding to the three straight line orbits of this system.\\

Note that on the straight line $r_1 = r_2$, the system reduces to \begin{equation}\begin{aligned}
    r_{1t}&=r_1^3 (\xi  + \eta ), \\
    r_{2t}&=r_2^3 (\eta  + \xi ).
\end{aligned}\end{equation} Thus, when $\xi + \eta >0$, the solutions tend away from the origin, and when $\xi + \eta <0$, the solutions tend towards the origin. In other words, when 
$$\gamma _3 < \frac29\left[\frac{1}{3|k_1^c|^2}+\frac{2}{|k_1^c-k_2^c|^2-|k_1^c|^2}+\frac{2}{|k_1^c+k_2^c|^2-|k_1|^2} \right]\gamma_2 ^2,$$ 
the solutions tend away from the origin, and when 
$$\gamma _3 > \frac29\left[\frac{1}{3|k_1^c|^2}+\frac{2}{|k_1^c-k_2^c|^2-|k_1^c|^2}+\frac{2}{|k_1^c+k_2^c|^2-|k_1|^2} \right]\gamma_2 ^2,$$
 the solutions tend towards the origin.

Note that on the straight line $r_1 = 0$, the system reduces to \begin{equation}
    \begin{aligned}
    r_{1t}&=0, \\
    r_{2t}&=\xi r_2^3. 
\end{aligned}
\end{equation} Thus, when $\xi >0$, the solutions tend away from the origin, and when $\xi <0$, the solutions tend towards the origin. In other words, when $\gamma_3 < \frac{2}{9|k_1^c|^2}\gamma_2^2$, the solutions tend away from the origin, and when $\gamma_3 > \frac{2}{9|k_1^c|^2}\gamma_2^2$, the solutions tend towards the origin.\\

Note that on the straight line $r_2 = 0$, the system reduces to \begin{equation}
    \begin{aligned}
    r_{1t}&=\xi r_1^3 ,\\
    r_{2t}&=0.
\end{aligned}
\end{equation}Thus, when $\xi >0$, the solutions tend away from the origin, and when $\xi <0$, the solutions tend towards the origin. In other words, when $\gamma_3 < \frac{2}{9|k_1^c|^2}\gamma_2^2$, the solutions tend away from the origin, and when $\gamma_3 > \frac{2}{9|k_1^c|^2}\gamma_2^2$, the solutions tend towards the origin.\\

By putting this together, when $\xi + \eta <0$ and $\xi <0$, or equivalently 
$$
\gamma_3 > \frac29 \max \left\{\frac{1}{3|k_1^c|^2}+\frac{2}{|k_1^c-k_2^c|^2-|k_1^c|^2}+\frac{2}{|k_1^c+k_2^c|^2-|k_1|^2}, \frac{1}{|k_1^c|^2}\right\}\gamma_2 ^2,
$$ 
solutions along all three of the straight lines mentioned above approach the origin, so the transition is Type I.\\

When $\xi + \eta >0$ and $\xi >0$, or equivalently $$
\gamma_3 < \frac29 \min \left\{\frac{1}{3|k_1^c|^2}+\frac{2}{|k_1^c-k_2^c|^2-|k_1^c|^2}+\frac{2}{|k_1^c+k_2^c|^2-|k_1|^2}, \frac{1}{|k_1^c|^2}\right\}\gamma_2 ^2,
$$  
solutions along all three of the straight lines mentioned above tend away from the origin, so the transition is Type II.\\

If  $$\frac{2}{9|k_1^c|^2}\gamma_2^2 < \gamma_3 <  \frac29 \left\{\frac{1}{3|k_1^c|^2}+\frac{2}{|k_1^c-k_2^c|^2-|k_1^c|^2}+\frac{2}{|k_1^c+k_2^c|^2-|k_1|^2}\right\}\gamma_2 ^2,$$ 
then $\xi < 0$ and $\xi + \eta >0$. This means that solutions along the line $r_1 = r_2$ tend away from the origin, but solutions along the lines $r_1 = 0$ and $r_2 = 0$ tend towards the origin, so the transition is Type II.\\

If 
$$\frac{2}{9|k_1^c|^2}\gamma_2^2 > \gamma_3 >  \frac29 \left\{\frac{1}{3|k_1^c|^2}+\frac{2}{|k_1^c-k_2^c|^2-|k_1^c|^2}+\frac{2}{|k_1^c+k_2^c|^2-|k_1|^2}\right\}\gamma_2 ^2,$$ 
then $\xi > 0$ and $\xi + \eta <0$. This means that solutions along the line $r_1 = r_2$ tend towards the origin, but solutions along the lines $r_1 = 0$ and $r_2 = 0$ tend away from the origin, so the transition is Type II.\\

Figure \ref{f5.1} shows graphs of the straight line orbits in \(r_1,r_2\) space. Notice that only Figure \ref{f5.1}(a) features solutions that approach the origin, leading to Type I transition. 
\begin{figure}
    \centering
    \includegraphics[width=0.5\linewidth]{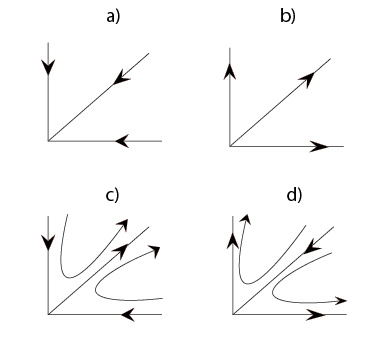}
  \caption{Straight line orbits for: (a) \(\xi+\eta<0\) and \(\xi<0\); (b) \(\xi+\eta>0\) and \(\xi>0\); (c) \(\xi+\eta>0\) and \(\xi<0\); (d) \(\xi+\eta<0\) and $\xi >0$.} \label{f5.1} 
\end{figure}
Theorem~\ref{tm4} is then proved. 

\subsection{Structure of the set of transition states}\label{s5.3}
 To study the detailed structure of the local attractor $\Sigma_\lambda$ in Theorem~\ref{tm4}, representing all transition states for $\lambda > \lambda_0$, we examine the nontrivial fixed points of the reduced system by solving \begin{equation}\begin{aligned}
    \beta r_1 &=-r_1 (\xi r_1^2 + \eta r_2^2), \\
    \beta r_2 &=-r_2 (\eta r_1^2 + \xi r_2^2).
\end{aligned}
\end{equation} 
The solutions are
\begin{equation}
\label{4.83}
\begin{aligned} 
p_1=&(r_1,r_2)=(0,\sqrt{\frac{-\beta}{\xi}}),\\ 
p_2=&(r_1,r_2)=(\sqrt{\frac{-\beta}{\xi}},0),\\
p_3=&(r_1,r_2)=(\sqrt{\frac{-\beta}{\xi + \eta}},\sqrt{\frac{-\beta}{\xi + \eta}}), \quad \text{ if $\xi^2 \neq \eta ^2$.}
\end{aligned} 
\end{equation}
Note that the general solution $u$ in $\Sigma_\lambda$ is given by \begin{align}
    u=y_1e_1+y_2e_2+y_3e_3+y_4e_4+\Phi(y,\lambda),
\end{align} with $\Phi=o(|y|)$. Therefore, the eigenfunctions $\sum_{k=1}^4 y_ke_k$ dictate the typical patterns of solutions, represented by $\Sigma_\lambda$, for $\lambda$ beyond the critical threshold $\lambda_0$. The Jacobian of the system at a fixed point $(r_1,r_2)$ is given by \begin{align}
    J=\begin{pmatrix} \beta +3\xi r_1^2 + \eta r_2^2 & 2\eta r_1r_2 \\ 2\eta r_1r_2 & \beta +\eta r_1^2 + 3\xi r_2^2 \end{pmatrix}.
\end{align} 

Consider the solution $(r_1,r_2)=(0,\sqrt{\frac{-\beta}{\xi}})$. It can be seen that the Jacobian calculated at this solution is $diag(\beta +\eta |\frac{\beta}{\xi}|,\beta +3\xi |\frac{\beta}{\xi}|)$. Thus, if $\beta >0$, then $-|\xi|<\eta$ implies that this solution is a saddle and $-|\xi|>\eta$ implies this solution is stable. \\

Consider the solution $(r_1,r_2)=(\sqrt{\frac{-\beta}{\xi}},0)$. It can be seen that the Jacobian calculated at this solution is $diag(\beta +3\xi |\frac{\beta}{\xi}| ,\beta + \eta |\frac{\beta }{\xi}|)$. Thus, if $\beta >0$, then $-|\xi|<\eta$ implies that this solution is a saddle and $-|\xi|>\eta$ implies this solution is stable.\\

Consider the solution $(r_1,r_2)=(\sqrt{\frac{-\beta}{\xi + \eta}},\sqrt{\frac{-\beta}{\xi + \eta}})$ with the condition that $\xi^2 \neq \eta ^2$. The Jacobian calculated at this solution is \begin{align}
    \begin{pmatrix} \beta + (3\xi + \eta )|\frac{\beta }{\xi + \eta}|  & 2\eta |\frac{ \beta }{\xi + \eta}| \\  2\eta |\frac{ \beta }{\xi + \eta}| &  \beta + (3\xi + \eta )|\frac{\beta }{\xi + \eta}|  \end{pmatrix}.
\end{align} It can be calculated that the eigenvalues of this matrix are $\beta - (\eta - 3 \xi )|\frac{\beta }{\xi + \eta}|$ and $\beta +3\beta (\frac{\eta + \xi}{\beta}*|\frac{\beta }{\xi + \eta}|)$. If $\beta >0$, then $\eta >\xi$ implies this solution is stable and $\eta <\xi$ implies this solution is a saddle. In conclusion, we have the following theorem:

\begin{theorem}[Structure of $\Sigma_\lambda$]\label{tm5.1} 
In the collapsed phase space $(r_1,r_2)$, the bifurcated attractor from case one in Theorem 4.1, $\Sigma_\lambda\approx S^3$, collapses to an arc in the first quadrant, $\Sigma_\lambda^r$, which contains three fixed points: $p_1$, $p_2$, $p_3$, given in (\ref{4.83}). For $j \in \{1,2,3\}$, $p_j$ generates the circles ($p_1$ and $p_2$) or torus ($p_3$) of steady states $(r_1^{(j)} e^{i\theta},r_2^{(j)} e^{i\theta})$, which are all contained in $\Sigma_\lambda$.
\end{theorem}

\subsection{Example: square lattice}\label{s5.4}
\-\quad\ Let $l_1=(\frac{2\pi}{50},0)$ and $l_2=(0,\frac{2\pi}{50})$. The scaling factor $50$ is chosen so that the patterns are easier to be visualized. Then, the dual lattice is spanned by the vectors $k_1 = (50,0)$ and $k_2 = (0,50)$. The critical vectors in this case are $k_1^c=k_1$ and $k_2^c=k_2$. Thus $|k_1^c|^2 = |k_2^c|^2 = 50$ and $|k_1^c + k_2^c|^2 = |k_1^c - k_2^c|^2 = 5000$. We also have that $\beta = 50\lambda - 2500$, $\xi = \frac{2}{3}\gamma _2^2 -150 \gamma_3$, and  $\eta = \frac{8}{99} \gamma_2^2-300\gamma_3$. From the theorem in section 6, when $\gamma_3 > \max \{\frac{37}{22275}, \frac{2}{450}\}\gamma_2 ^2 = \frac{2}{450}\gamma_2 ^2$, all straight line orbits tend towards the origin and the transition is Type I. When $\frac{37}{22275}\gamma_2^2<\gamma_3 < \frac{2}{450}\gamma_2 ^2$, solutions along the straight line $r_1=r_2$ tend away from zero, but solutions along $r_1=0$ and $r_2 = 0$ tend towards the origin and the transition is Type II. When $\gamma_3<\frac{37}{22275}\gamma_2^2$, all straight line orbits tend away from the origin and the transition is Type II.\\

The three stationary solutions are
\begin{equation}
    \begin{aligned}
        p_1 &=(r_1,r_2)=\left(0,\sqrt{\frac{7500-150\lambda}{2\gamma _2^2 -450 \gamma _3}}\right),\\ 
        p_2&=(r_1,r_2)=\left(\sqrt{\frac{7500-150\lambda}{2\gamma _2^2 -450 \gamma _3}},0\right), \\
        p_3&=(r_1,r_2)=\left(\sqrt{\frac{247500-4950\lambda}{74\gamma _2^2 -44550 \gamma _3}},\sqrt{\frac{247500-4950\lambda}{74\gamma _2^2 -44550 \gamma _3}}\right).
    \end{aligned}\label{5.40}
\end{equation} 

The trivial solution is unstable when the control parameter exceeds the critical threshold, i.e. when $\lambda >50$. Now let $\lambda >50$. For the first and second solutions, $\frac{26}{27}\gamma_2^2 <\gamma_3 <\frac{22}{9} \gamma_2^2$ implies the solutions are stable, and else are saddles. For the third solution, $\frac{26}{27}\gamma_2^2 <\gamma_3 <\frac{22}{9}\gamma_2^2$ implies the solution is stable, and else is a saddle.\\

{\bf Solution \(p_1\).} Recalling that \(a_1^2+a_2^2=r_1^2\), \(a_3^2+a_4^2=r_2^2\), and further that \(y_1=a_1+a_2i\), \(y_3=a_3+a_4i\), we see that our stationary solutions to the reduced system are radial. From (\ref{new4.2}), we can write the solutions as \(u(x,t)=y_1e_1+\bar{y_1}\bar{e_1}+y_3e_3+\bar{y_3}\bar{e_3}\) where the stable component is of little significance anymore and can be dropped. In this case, \(r_1=0\) so the solution becomes \(u=y_3e_3+\bar{y_3}\bar{e_3}\). Expanding and noting that \(k_2^c=(0,50)\) and \(e_3=e^{ik_2^cx}\), we have 
\begin{equation}
\begin{aligned} 
u(x,t)&=(a_3+ia_4)(\cos(k_2^c \cdot x)\\&+i\sin(k_2^c\cdot x))+(a_3-ia_4)(\cos(k_2^c\cdot x)-i\sin(k_2^c\cdot x))\\&=2(a_3\cos(50x_2)-a_4\sin(50x_2)),
\end{aligned}
\end{equation}
where \(x=(x_1,x_2)\). As \((a_3,a_4)\) run along the circle \(a_3^2+a_4^2=r_2^2\), a set of stationary solutions is generated in \((x,t)\)-space. The principle exchange of stability guarantees that patterns in the form of solutions to (\ref{2.5}) emerge as \(\lambda\) crosses the critical threshold \(\lambda_0=50\). \\

Figure \ref{f5.2} shows a graph of the stationary solution \( p_1\), when \(\lambda=50.1\), \(\gamma_2=1\), and \(\gamma_3=\frac{17}{450}\). In this case, \((r_1,r_2)=(0,1)\) and \(\gamma_3<\frac{26}{27}\gamma_2^2\) so that the solution is a saddle and we have chosen the point \((a_3,a_4)=(\frac{\sqrt{2}}{2},\frac{\sqrt{2}}{2})\) on the circle of solutions. The characteristic patterns for this stationary solution are horizontal rolls. Graphically, we see that the size of the domain is responsible for the amount of rolls within the square, and the patterns shift vertically as the parameters $(a_1,a_2)$ run along the unit circle. 

\medskip

{\bf Solution $p_2$.}  In this case,  the solution can be written as
\begin{equation}
\begin{aligned}
 u(x,t)&=(a_1+ia_2)(\cos(k_1^c \cdot x)\\&+i\sin(k_1^c\cdot x))+(a_1-ia_2)(\cos(k_1^c\cdot x)-i\sin(k_1^c\cdot x))\\ &= 2(a_1\cos(50x_1)-a_2\sin(50x_1)),
 \end{aligned}
\end{equation}
which also produces rolls, however this time horizontal. Figure \ref{f5.3} shows a graph of this solution for the same choice of constants and parameters used in Figure \ref{f5.3}.
\begin{figure}
\centering
    \includegraphics[width=0.5\linewidth]{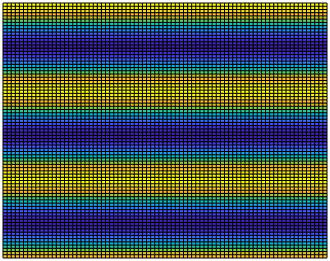}
    \caption{Horizontal rolls exhibited by the stationary solution \\$u(x,t)=\sqrt{2}\cos(50x_2)-\sqrt{2}\sin(50x_2)$.}
    \label{f5.2}
\end{figure}
\begin{figure}
    \centering
    \includegraphics[width=0.5\linewidth]{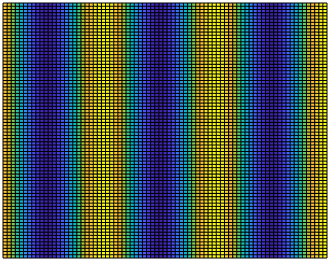}
    \caption{Vertical rolls exhibited by the stationary solution 
    \\ $u(x,t)=\sqrt{2}\cos(50x_1)-\sqrt{2}\sin(50x_1)$.}
    \label{f5.3}
\end{figure}

{\bf Solution $p_3$.} Finally, the solution for $p_3$ can be written as the sum of the previous solutions
\begin{align} u(x,t)=2(a_1\cos(x_1)-a_2\sin(x_1)+a_3\cos(x_2)-a_4\sin(x_2)), \end{align}
where both \((a_1,a_2)\) and \((a_3,a_4)\) run along different circles centered at the origin with radii \(r_1=r_2\).\\

Figure \ref{f5.4} shows the graph of this solution for \(\lambda=50.1\), \(\gamma_2=1\), \(\gamma_3=\frac{569}{445500}\). In this case, \((r_1,r_2)=(1,1)\), and \(\gamma_3>\frac{26}{27}\gamma_2^2\) so that the solution is stable. Here, we have chosen \((a_1,a_2)=(\frac{\sqrt{2}}{2},\frac{\sqrt{2}}{2})\) and \((a_3,a_4)=(\frac{\sqrt{2}}{2}, \frac{\sqrt{2}}{2})\) as the two points on the unit circle. Notice here that the characteristic patterns are circles and that they are arranged in a square-like fashion throughout the lattice. We will encounter another pattern that exhibits circles packed in a different manner, namely hexagonally-packed circles. 
\begin{figure}
    \centering
    \includegraphics[width=0.5\linewidth]{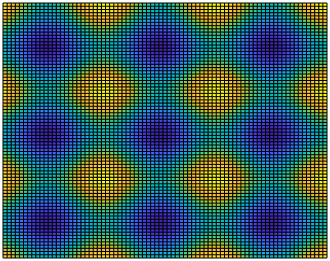}
    \caption{Square-packed circles exhibited by stationary solution $u(x,t)=\sqrt{2}[ \cos(50x_1)-\sin(50x_1)+\cos(50x_2)-\sin(50x_2)]$.}
    \label{f5.4}
\end{figure}

\section{Multiplicity Two Case}

%\section{Proof of Theorem~\ref{tm2}}
\subsection{Dynamical transition theorem}
Consider the same situation as above but assume that the pairs of $(n_1,n_2)$ that minimize $|n_1k_1+n_2k_2|^2$ are $(n_1^c,n_2^c)$ and $(-n_1^c,-n_2^c)$. Equivalently, the cardinality of $S$ is two. In this case, $\beta_{n_1n_2}(\lambda_0)<0$ if $(n_1,n_2)\in \Z \times \Z \setminus \{(0,0),(n_1^c,n_2^c),(-n_1^c,-n_2^c)\}$. Thus, with critical value $\lambda_0$, the eigenvalue $-|n_1^c k_1 +n_2^c k_2|^2 (|n_1^c k_1 +n_2^c k_2|^2 -\lambda) = -|k_c|^2(|k_c|^2-\lambda)$, has multiplicity two.   We have
\begin{equation}
E_1^\lambda =  \text{span} \{e_1 = e^{i(k_c \cd x)},\quad e_2 = e^{-i(k_c \cd x)}\}, \quad E_2^\lambda=\overline{\text{span}\{e_3(\lambda), ...\}}.\end{equation}

\begin{theorem}[Transition Types for Multiplicity Two]\label{tm2}
    Assume the multiplicity of $\beta_1$ is two at $\lambda=\lambda_0=|k_c|^2$. The following are true:
    
    \begin{enumerate}
    
 \item
  If 
  $$\gamma _3 > \frac{2}{9|k_c|^2}\gamma_2^2,$$ 
  as $\lambda$ crosses $\lambda_0$, the system undergoes a continuous dynamical transition (Type I) to $\Sigma_\lambda\approx S^1$ consisting of a circle of steady-states. 
  
  \item  If  $$\gamma _3 < \frac{2}{9|k_c|^2}\gamma_2^2, $$ 
 as $\lambda$ crosses $\lambda_0$,  the system undergoes a catastrophic  dynamical transition (Type II).
  \end{enumerate}
\end{theorem}

An example of this is when the lattice is spanned by the vectors $l_1=(\frac{\pi}{25},-\frac{\sqrt{3}\pi}{25})$ and $l_2=(0,\frac{2\pi}{25})$, which will be discussed in section 5.2. The solution can thus be written as \begin{align} 
u(x,t)=y_1e_1+y_2e_2+z, \qquad z\in E_2^\lambda. 
\end{align}  
We can show that the center manifold function up to higher order terms is  \begin{align}\phi(x) = \frac{-\gamma_2 y_1^2}{4|k_c|^2-\lambda}e^{2ik_c \cd x}-\frac{\gamma_2 y_2^2}{4|k_c|^2-\lambda}e^{-2ik_c \cd x}.\end{align} 
Then it can be calculated that the reduced equations for this system are \begin{equation}
    \begin{aligned}
    y_{1t}=-|k_c|^2(|k_c|^2 - \lambda)y_1 +\frac{2|k_c|^2\gamma_2^2y_1^2y_2}{4|k_c|^2-\lambda}-3|k_c|^2y_1^2y_2\gamma_3 +o(3),\\
    y_{2t}=-|k_c|^2(|k_c|^2 - \lambda)y_2 +\frac{2|k_c|^2\gamma_2^2y_1y_2^2}{4|k_c|^2-\lambda}-3|k_c|^2y_1y_2^2\gamma_3 +o(3). \label{5.4}  
\end{aligned}
\end{equation}

Let  
\begin{equation}
    \begin{aligned}
   &   y_1 = a_1 + a_2 i, \qquad  y_2 = a_1 - a_2 i, \\
    \eta &= \frac{2}{3}\gamma_2^2 -3|k_c|^2 \gamma_3, 
    \end{aligned}
\end{equation}
then the system given by (\ref{5.4}) can be rewritten as \begin{equation}
    \begin{aligned}
    a_{1t}=\eta a_1 (a_1^2 + a_2^2),\\
    a_{2t}=\eta a_2 (a_1^2 + a_2^2).
\end{aligned} 
\end{equation}
    By analyzing this system, it can be seen that all solutions tend towards the origin when $\eta <0$ and tend away from the origin when $\eta >0$. Thus, the transition is Type I when $\eta <0$ and Type II when $\eta >0$.\\

By using the approximative system \begin{align}
    a_{1t}=\beta a_1 + \eta a_1 (a_1^2 + a_2^2),\\
    a_{2t}=\beta a_2 + \eta a_2 (a_1^2 + a_2^2),
\end{align} and letting $a_1^2 + a_2^2 = r^2$, this system can be rewritten as \begin{align}
    r_t=\beta r + \eta r^3.
\end{align} The nontrivial equilibrium of this system is $r=\sqrt{\frac{-\beta}{\eta}}$. The Jacobian of this system at a fixed point $r$ is \begin{align}
    J=\begin{pmatrix} \beta +3\eta r^2  \end{pmatrix}.
\end{align} When $r=\sqrt{\frac{-\beta}{\eta}}$, the eigenvalue of the Jacobian is $\beta +3\eta |\frac{\beta}{\eta}|$. If $\beta >0$, then $\eta<0$ must be true, which implies that this solution will be stable.
Stationary solutions in this case are given by 
\begin{equation}\begin{aligned}u(x,t)&=y_1e_1+\bar{y_1}\bar{e_1}\\&=(a_1+ia_2)(\cos(k_c\cdot x)+i\sin(k_c\cdot x))\\&+(a_1-ia_2)(\cos(k_c\cdot x)-i\sin(k_c\cdot x))\\&=2(a_1\cos(k_c\cdot x)-a_2\sin(k_c\cdot x)),\end{aligned}\end{equation}
where \((a_1,a_2)\) run along the circle \(a_1^2+a_2^2=r^2\).  Note that solutions depend solely on the two critical vectors of the lattice in which the magnitudes are least, and that the spanning vectors play no direct role besides specifying the domain.

\subsection{Example: roll patterns on parallelogram}
 Let $l_1=(\frac{\pi}{25},-\frac{\sqrt{3}\pi}{25})$ and $l_2=(0,\frac{2\pi}{25})$. Then the dual lattice is spanned by the vectors where $k_1 = (50,0)$ and $k_2 = (25\sqrt{3},25)$. It can be shown that $|k_1|^2 = |k_2|^2 = 2500$, $|k_1 + k_2|^2 = 3125 + 1250\sqrt{3}$, and $|k_1 - k_2|^2 = 5000 - 2500\sqrt{3}$. The critical points of the lattice are \(k_2-k_1\) and \(k_1-k_2\) and so we will use the analysis outlined in the section dealing lwith multiplicity two. In this case, it can be shown that $\beta = -(5000-2500\sqrt{3})(5000-2500\sqrt{3}-\lambda)$ and \(\eta = \frac{2}{3}\gamma_2^2-3(5000-2500\sqrt{3})\gamma_3\). Using the theorem we proved in the previous subsection,  when \(\gamma_3 > \frac{2}{9(5000-2500\sqrt{3})}\gamma_2^2\) , all straight line orbits tend towards the origin and the transition is Type I. When \(\gamma_3 < \frac{2}{9(5000-2500\sqrt{3})}\gamma_2^2\) all straight line orbits tend away from the origin and the transition is Type II.\\

The nontrivial stationary solution is \(r=\sqrt{\frac{(2-\sqrt{3})(2-\sqrt{3}-\lambda)}{\frac{2}{3}\gamma_2^2-3(2-\sqrt{3})\gamma_3}}\). The non-trivial solution is always stable for \(\lambda>\lambda_0\). The solution can be written as
\begin{align}u(x,t)=2(a_1\cos(k_c\cdot x)-a_2\sin(k_c\cdot x)),\end{align} where \(k_c=k_2-k_1=(\frac{\sqrt{3}}{2}-1, \frac{1}{2})\) and \((a_1,a_2)\) run along the circle \(a_1^2+a_2^2=r^2\). 
Figure \ref{f6.1} shows a graph of the solution for \(\lambda=2\), \(\gamma_2=1\), and \(\gamma_3=-\frac{7-6\sqrt{3}}{18-9\sqrt{3}}\), in which case \(r=1\). The parameters \((a_1,a_2)\) are evaluated at the point \((\frac{\sqrt{3}}{2},\frac{1}{2})\). Notice the characteristic patterns are horizontal rolls similar to the square case. 
\begin{figure}
    \centering
    \includegraphics[width=0.5\linewidth]{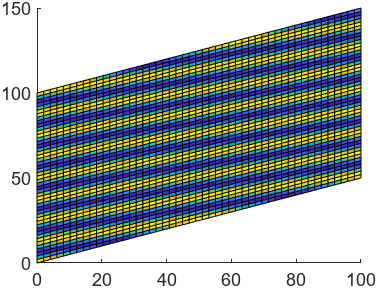}
    \caption{Stationary solution $r=1$.}
    \label{f6.1}
\end{figure}

\section{Multiplicity Six Case}\label{s7}
\subsection{Dynamic transition theorem}
Consider the same situation as above but assume that $\#S = 6$. Then
\begin{align}\label{7.1}
    S=\{(n_1^c,n_2^c),(-n_1^c,-n_2^c),(n_3^c,n_4^c),(-n_3^c,-n_4^c),(n_5^c,n_6^c),(-n_5^c,-n_6^c)\}.
\end{align} 
Thus, we have $\lambda_0 = |k_1^c|^2$, and $\beta_{n_1n_2}(\lambda_0)<0$ for all $(n_1,n_2)\in \Z ^2 \setminus (S\cup \{(0,0)\})$. Thus, with critical value $\lambda_0$, the eigenvalue 
\begin{equation}
\beta_{n_1^cn_2^c}(\lambda) = -|n_1^c k_1 +n_2^c k_2|^2 (|n_1^c k_1 +n_2^c k_2|^2 -\lambda) = -|k_c|^2(|k_c|^2-\lambda),
\end{equation} 
has multiplicity six with 
\begin{equation}
\begin{aligned} 
e_1 &= e^{i(k_1^c \cd x)},&e_2 &= e^{-i(k_1^c \cd x)},&e_3 &= e^{i(k_2^c \cd x)},\\e_4 &= e^{-i(k_2^c \cd x)},&e_5 &= e^{i(k_3^c \cd x)},&e_6 &= e^{-i(k_3^c \cd x)},\\
E_1^\lambda &= \text{span}\{e_1,...,e_6\},& E_2^\lambda &= \overline{ \text{span} \{e_7,e_8,...\} }.
\end{aligned}
\end{equation} 
The solution can thus be written as 
\begin{align}
 u(x,t)=\sum_{i=1}^6y_ie_i+z\in E_1^\lambda \bigoplus E_2^\lambda,
\end{align} 
where $z\in E_2^\lambda$ is the stable component. By similar computation, the center manifold function up to higher order terms is given by $\phi(x) = \sum_{i=7}^{24} \phi_i e_i$. 
Using the notation that 
\begin{equation}
\begin{aligned}
e_7&= e^{2ik_1^c\cd x}, &e_8&= e^{-2ik_1^c\cd x}, &e_9&= e^{2ik_2^c\cd x},\\
e_{10}&= e^{-2ik_2^c\cd x} ,&e_{11}&= e^{2ik_3^c\cd x},&e_{12}&= e^{-2ik_3^c\cd x},\\
e_{13}&= e^{i(k_1^c+k_2^c)\cd x},&e_{14}&= e^{-i(k_1^c+k_2^c)\cd x},&e_{15}&= e^{i(k_1^c+k_3^c)\cd x},\\
e_{16}&= e^{-i(k_1^c+k_3^c)\cd x},&e_{17}&= e^{i(k_2^c+k_3^c)\cd x},&e_{18}&= e^{-i(k_2^c+k_3^c)\cd x},\\
e_{19}&= e^{i(k_1^c-k_2^c)\cd x},&e_{20}&= e^{-i(k_1^c-k_2^c)\cd x},&e_{21}&= e^{i(k_1^c-k_3^c)\cd x},\\
e_{22}&= e^{-i(k_1^c-k_3^c)\cd x},&e_{23}&= e^{i(k_2^c-k_3^c)\cd x},&e_{24}&= e^{-i(k_2^c-k_3^c)\cd x},
\end{aligned}
\end{equation} 
it can be calculated that the coefficients of the manifold are 
\begin{align*}
\phi_7&= \frac{-\gamma_2 y_1^2}{4|k_1^c|^2-\lambda},\quad &\phi_8&= \frac{-\gamma_2 y_2^2}{4|k_1^c|^2-\lambda},\\
\phi_9&= \frac{-\gamma_2 y_3^2}{4|k_1^c|^2-\lambda},\quad &\phi_{10}&= \frac{-\gamma_2 y_4^2}{4|k_1^c|^2-\lambda},\\
\phi_{11}&= \frac{-\gamma_2 y_5^2}{4|k_1^c|^2-\lambda},\quad &\phi_{12}&= \frac{-\gamma_2 y_6^2}{4|k_1^c|^2-\lambda},\\
\phi_{13}&= \frac{-2\gamma_2 y_1y_3}{|k_1^c+k_2^c|^2-\lambda},\quad &\phi_{14}&= \frac{-2\gamma_2 y_2y_4}{|k_1^c+k_2^c|^2-\lambda},\\
\phi_{15}&= \frac{-2\gamma_2 y_1y_5}{|k_1^c+k_3^c|^2-\lambda},\quad &\phi_{16}&= \frac{-2\gamma_2 y_2y_6}{|k_1^c+k_3^c|^2-\lambda},\\
\phi_{17}&= \frac{-2\gamma_2 y_3y_5}{|k_2^c+k_3^c|^2-\lambda},\quad &\phi_{18}&= \frac{-2\gamma_2 y_4y_6}{|k_2^c+k_3^c|^2-\lambda},\\
\phi_{19}&= \frac{-2\gamma_2 y_1y_4}{|k_1^c-k_2^c|^2-\lambda},\quad &\phi_{20}&= \frac{-2\gamma_2 y_2y_3}{|k_1^c-k_2^c|^2-\lambda},\\
\phi_{21}&= \frac{-2\gamma_2 y_1y_6}{|k_1^c-k_3^c|^2-\lambda},\quad &\phi_{22}&= \frac{-2\gamma_2 y_2y_5}{|k_1^c-k_3^c|^2-\lambda},\\
\phi_{23}&= \frac{-2\gamma_2 y_3y_6}{|k_2^c-k_3^c|^2-\lambda},\quad &\phi_{24}&= \frac{-2\gamma_2 y_4y_5}{|k_2^c-k_3^c|^2-\lambda}.
\end{align*}
Let 
\begin{align*}
    (y_1, y_2) = (a_1 + a_2 i,  a_1 - a_2 i),
    (y_3, y_4)  = (a_3 + a_4 i,  a_3 - a_4 i),
    (y_5, y_6) = (a_5 + a_6 i,  a_5 - a_6 i).
\end{align*}

Set \begin{align*}
    & D_{12}^\pm = \frac{1}{|k_1^c\pm k_2^c|^2 - |k_1^c|^2},
   \quad  D_{13}^\pm = \frac{1}{|k_1^c\pm k_3^c|^2 - |k_1^c|^2},
  \quad   D_{23}^\pm = \frac{1}{|k_2^c\pm k_3^c|^2 - |k_1^c|^2},
  \\
 &   \xi = -3|k_1^c|^2 \gamma_3 + \frac{2}{3}\gamma_2^2,\\
  &  \eta = -6|k_1^c|^2 \gamma_3 + |k_1^c|^2(4D_{12}^-+4D_{12}^+)\gamma_2^2,\\
&    \chi = -6|k_1^c|^2 \gamma_3 + |k_1^c|^2(4D_{13}^-+4D_{13}^+)\gamma_2^2,\\
&   \omega = -6|k_1^c|^2 \gamma_3 + |k_1^c|^2(4D_{23}^-+4D_{23}^+)\gamma_2^2.
\end{align*}
 
Then, the reduced system on the center manifold can be rewritten as \begin{equation}
\begin{aligned}
    a_{1t}= \beta a_1 + a_1 (\xi (a_1^2 + a_2^2) +\eta (a_3^2 + a_4^2)+\chi (a_5^2 + a_6^2))+o(3),\\
    a_{2t}= \beta a_2 + a_2 (\xi (a_1^2 + a_2^2) +\eta (a_3^2 + a_4^2)+\chi (a_5^2 + a_6^2))+o(3),\\
    a_{3t}= \beta a_3 +a_3 (\eta (a_1^2 + a_2^2) +\xi (a_3^2 + a_4^2)+\omega (a_5^2 + a_6^2))+o(3),\\
    a_{4t}= \beta a_4 +a_4 (\eta (a_1^2 + a_2^2) +\xi (a_3^2 + a_4^2)+\omega (a_5^2 + a_6^2))+o(3),\\
    a_{5t}= \beta a_5 +a_5 (\chi (a_1^2 + a_2^2) +\omega (a_3^2 + a_4^2)+\xi (a_5^2 + a_6^2))+o(3),\\
    a_{6t}= \beta a_6 +a_6 (\chi (a_1^2 + a_2^2) +\omega (a_3^2 + a_4^2)+\xi (a_5^2 + a_6^2))+o(3).
\end{aligned}
\end{equation}
It can be calculated that 
\begin{align*}
    \xi  >0 &\iff \gamma_3 < \frac{2}{9|k_1^c|^2}\gamma_2^2,\\
    \xi <0 &\iff  \gamma_3 > \frac{2}{9|k_1^c|^2}\gamma_2^2,\\
    \xi  +\omega > 0 &\iff \gamma_3 < (\frac{2}{27|k_1^c|^2} + \frac{4}{9}D_{23}^- + \frac{4}{9}D_{23}^+)\gamma_2^2,\\
    \xi +\omega < 0 &\iff \gamma_3 > (\frac{2}{27|k_1^c|^2} + \frac{4}{9}D_{23}^- + \frac{4}{9}D_{23}^+)\gamma_2^2,\\
    \xi  +\chi > 0 &\iff \gamma_3 < (\frac{2}{27|k_1^c|^2} + \frac{4}{9}D_{13}^- + \frac{4}{9}D_{13}^+)\gamma_2^2,\\
    \xi +\chi < 0 &\iff \gamma_3 > (\frac{2}{27|k_1^c|^2} + \frac{4}{9}D_{13}^- + \frac{4}{9}D_{13}^+)\gamma_2^2,\\
    \xi  +\eta > 0 &\iff \gamma_3 < (\frac{2}{27|k_1^c|^2} + \frac{4}{9}D_{12}^- + \frac{4}{9}D_{12}^+)\gamma_2^2,\\
    \xi +\eta < 0 &\iff \gamma_3 > (\frac{2}{27|k_1^c|^2} + \frac{4}{9}D_{12}^- + \frac{4}{9}D_{12}^-D_{12}^+)\gamma_2^2,
    \end{align*}
    \begin{align*}
    \xi +\eta +\chi > 0 &\iff \gamma_3 < (\frac{2}{45|k_1^c|^2} + \frac{4}{15}D_{12}^- + \frac{4}{15}D_{12}^+  + \frac{4}{15}D_{13}^- + \frac{4}{15}D_{13}^+)\gamma_2^2,\\
    \xi +\eta +\chi < 0 &\iff \gamma_3 > (\frac{2}{45|k_1^c|^2} + \frac{4}{15}D_{12}^- + \frac{4}{15}D_{12}^+ + \frac{4}{15}D_{13}^- + \frac{4}{15}D_{13}^+)\gamma_2^2,\\
    \eta +\xi +\omega > 0 &\iff \gamma_3 < (\frac{2}{45|k_1^c|^2} + \frac{4}{15}D_{12}^- + \frac{4}{15}D_{12}^+  + \frac{4}{15}D_{23}^- + \frac{4}{15}D_{23}^+)\gamma_2^2,\\
    \eta +\xi +\omega < 0 &\iff \gamma_3 > (\frac{2}{45|k_1^c|^2} + \frac{4}{15}D_{12}^- + \frac{4}{15}D_{12}^+ + \frac{4}{15}D_{23}^- + \frac{4}{15}D_{23}^+)\gamma_2^2,\\
    \chi +\omega +\xi > 0 &\iff \gamma_3 < (\frac{2}{45|k_1^c|^2} + \frac{4}{15}D_{13}^- + \frac{4}{15}D_{13}^+ + \frac{4}{15}D_{23}^- + \frac{4}{15}D_{23}^+)\gamma_2^2,\\
    \chi +\omega +\xi < 0 &\iff \gamma_3 > (\frac{2}{45|k_1^c|^2} + \frac{4}{15}D_{13}^- + \frac{4}{15}D_{13}^+ + \frac{4}{15}D_{23}^- + \frac{4}{15}D_{23}^+)\gamma_2^2.
\end{align*}
By letting $r_1^2 = a_1^2 + a_2^2$, $r_2^2 = a_3^2 + a_4^2$, and $r_3^2 = a_5^2 + a_6^2$, the reduced system becomes 
\begin{equation}
\label{7.7}
\begin{aligned}
    r_{1t}= r_1 (\xi r_1^2 +\eta r_2^2+\chi r_3^2),\\
    r_{2t}= r_2 (\eta r_1^2 +\xi r_2^2+\omega r_3^2),\\
    r_{3t}= r_3 (\chi r_1^2 +\omega r_2^2+\xi r_3^2).
\end{aligned} 
\end{equation}
The straight lines corresponding to this system are 
\begin{align*}
  & l_1: \quad  r_1=r_2=0;\\
   & l_2: \quad   r_1=r_3=0;\\
    & l_3: \quad  r_2=r_3=0;\\
  & l_4: \quad    r_1=0 \text{ and } r_2=r_3;\\
  & l_5: \quad    r_2=0 \text{ and } r_1=r_3;\\
   & l_6: \quad   r_3=0 \text{ and } r_1=r_2;\\
  & l_7: \quad    r_1=r_2=r_3.
\end{align*} 
On these straight lines, the systems that emerge are respectively
\begin{align*}
& l_1: \quad 
  r_{3t}=\xi r_3^3;\\
    &  l_2: \quad 
    r_{2t}=\xi r_2^3;\\
    & l_3: \quad 
    r_{1t}=\xi r_1^3;\\
 & l_4: \quad 
    \begin{cases}
    r_{2t}&= r_2^3 (\xi + \omega),\\
    r_{3t}&= r_3^3 (\xi + \omega);\\
    \end{cases}\\
 & l_5: \quad \begin{cases}
    r_{1t}&= r_1^3 (\xi + \chi),\\
    r_{3t}&= r_3^3 (\xi + \chi);\\
    \end{cases}\\
 & l_6: \quad \begin{cases}
    r_{1t}&= r_1^3 (\xi + \eta),\\
    r_{2t}&= r_2^3 (\xi + \eta);\\
    \end{cases}\\
 & l_7: \quad \begin{cases}
    r_{1t}&= r_1^3 (\xi + \eta +\chi),\\
    r_{2t}&= r_2^3 (\xi + \eta +\omega),\\
    r_{3t}&= r_3^3 (\xi + \chi + \omega).
    \end{cases}
\end{align*} 
Thus, the only way for all solutions to go to zero along these straight lines is for all of the coefficients to be negative. Let 
\begin{align}
\label{7.8}A=\max \left\{\frac{\xi}{\gamma_2^2},\frac{\xi + \omega}{\gamma_2^2},\frac{\xi + \chi}{\gamma_2^2},\frac{\xi + \eta}{\gamma_2^2},\frac{\xi + \eta +\chi}{\gamma_2^2},\frac{\xi + \eta +\omega}{\gamma_2^2},\frac{\xi + \chi +\omega}{\gamma_2^2}\right\},
\end{align} 
or equivalently, 
\begin{align}
 \label{7.9}
    A= & \max \{\frac{2}{9|k_1^c|^2}, (\frac{2}{27|k_1^c|^2} + \frac{4}{9}D_{23}^- + \frac{4}{9}D_{23}^+), (\frac{2}{27|k_1^c|^2} + \frac{4}{9}D_{13}^- + \frac{4}{9}D_{13}^+),\\ 
    \nonumber &(\frac{2}{27|k_1^c|^2} + \frac{4}{9}D_{12}^- + \frac{4}{9}D_{12}^+),\\
    \nonumber & (\frac{2}{45|k_1^c|^2} + \frac{4}{15}D_{12}^- + \frac{4}{15}D_{12}^+  + \frac{4}{15}D_{13}^- + \frac{4}{15}D_{13}^+),\\
    \nonumber & (\frac{2}{45|k_1^c|^2} + \frac{4}{15}D_{12}^- + \frac{4}{15}D_{12}^+  + \frac{4}{15}D_{23}^- + \frac{4}{15}D_{23}^+),\\
    \nonumber &(\frac{2}{45|k_1^c|^2} + \frac{4}{15}D_{13}^- + \frac{4}{15}D_{13}^+  + \frac{4}{15}D_{23}^- + \frac{4}{15}D_{23}^+) \}.
\end{align}
From this, it can be seen that the transition is Type I when $\gamma _3 > A \gamma_2^2$ and Type II when $\gamma _3 <  A \gamma_2^2$. This can be stated in the following theorem.

\begin{theorem}[Transition Types for Multiplicity Six Case]\label{tm7.1}
   Suppose $k_1^c$, $k_2^c$, and $k_3^c$ are defined as previously and $k_3^c = a k_1^c + b k_2^c$ such that $a,b \neq 0$ and $(a,b) \in \Z^2 \setminus \{(1,1),(-1,1),(-1,-1),(1,-1)\}$. Then, the following assertions hold true: 
   \begin{itemize}
   
   \item[1).] if 
   $$\gamma _3 >  A \gamma_2^2,$$
   then as $\lambda$ crosses $\lambda_0$, the system undergoes a continuous (Type I) transition to $\Sigma_\lambda$ homological to $S^5$, 
   
   \item[2).] if $$\gamma _3 <  A \gamma_2^2,$$
   then as $\lambda$ crosses $\lambda_0$, 
  then the system  undergoes a jump (Type II) transition.
  \end{itemize}
\end{theorem}

\subsection{Structure of the set of transition states}\label{s7.2}
By using the approximative system 
\begin{equation}\label{7.10}
\begin{aligned}
    r_{1t}= \beta r_1 + r_1 (\xi r_1^2 +\eta r_2^2+\chi r_3^2)+o(3),\\
    r_{2t}= \beta r_2 + r_2 (\eta r_1^2 +\xi r_2^2+\omega r_3^2)+o(3),\\
    r_{3t}= \beta r_3+ r_3 (\chi r_1^2 +\omega r_2^2+\xi r_3^2)+o(3),
\end{aligned} 
\end{equation}
the nontrivial equilibria of this system can be calculated to be \begin{align*}
    p_1:\quad (r_1,r_2,r_3)&=(\sqrt{\frac{-\beta}{\xi}},0,0),\\
    p_2:\quad (r_1,r_2,r_3)&=(0,\sqrt{\frac{-\beta}{\xi}},0),\\
    p_3:\quad (r_1,r_2,r_3)&=(0,0,\sqrt{\frac{-\beta}{\xi}}),\\
    p_4:\quad (r_1,r_2,r_3)&=(0,\sqrt{\frac{-\beta}{\omega + \xi}},\sqrt{\frac{-\beta}{\omega + \xi}}),\\
    p_5:\quad (r_1,r_2,r_3)&=(\sqrt{\frac{-\beta}{\chi + \xi}},0,\sqrt{\frac{-\beta}{\chi + \xi}}),\\
    p_6:\quad (r_1,r_2,r_3)&=(\sqrt{\frac{-\beta}{\eta + \xi}},\sqrt{\frac{-\beta}{\eta + \xi}},0),\\
    p_7:\quad (r_1,r_2,r_3)&=(\sqrt{\frac{\beta (\xi - \omega)(-\eta +\xi -\chi +\omega)}{\eta^2 \xi - 2\eta \chi \omega + \xi (-\xi^2 +\chi^2 +\omega^2)}},\\&\sqrt{\frac{\beta (\xi - \chi)(-\eta +\xi +\chi -\omega)}{\eta^2 \xi - 2\eta \chi \omega + \xi (-\xi^2 +\chi^2 +\omega^2)}},\\&\sqrt{\frac{\beta (\xi - \eta)(\eta +\xi -\chi -\omega)}{\eta^2 \xi - 2\eta \chi \omega + \xi (-\xi^2 +\chi^2 +\omega^2)}}).
\end{align*}
The Jacobian of the reduced system at a fixed point $(r_1,r_2,r_3)$ is 
\begin{align*}
    J=\begin{pmatrix} \beta + 3\xi r_1^2 + \eta r_2^2 + \chi r_3^2 & 2\eta r_1r_2  & 2\chi r_1r_3 \\ 2\eta r_1 r_2  & \beta + \eta r_1^2 + 3\xi r_2^2 + \omega r_3^2 & 2\omega r_2r_3 \\ 2\chi r_1 r_3  & 2\omega r_2r_3 & \beta + \chi r_1^2 + \omega r_2^2 + 3\xi r_3^2   \end{pmatrix} .\end{align*}
    The stability of each solution can be determined from calculating the eigenvalues of this matrix at each equilibria, such as in the example in the next section.

\begin{theorem}[Structure of  $\Sigma_\lambda$]\label{tm7.2}
Under a Type I transition, the bifurcated attractor $\Sigma_\lambda$ from Theorem \ref{tm7.1} is homological to $S^5$, with $p_1$, $p_2$, and $p_3$ correspond to circles of solutions; $p_4$, $p_5$, and $p_6$ correspond to torii of solutions; and $p_7$ corresponding to an $S^1 \times S^1 \times S^1$ surface. 
    
\end{theorem}

 \subsection{Example: roll patterns}\label{s7.3}
 Let $l_1=(2\pi,\frac{7\sqrt{15}\pi}{15})$ and $l_2=(0,\frac{8\sqrt{15}\pi}{15})$. Then the dual lattice is spanned by $k_1=(1,0)$ and $k_2=(-\frac{7}{8},\frac{\sqrt{15}}{4})$. Note that $|k_1-k_2|^2=\frac{285}{64}$, $|k_1+k_2|^2=\frac{61}{64}$. The critical points of the lattice are thus $k_1$, $-k_1$, $k_2$, $-k_2$, $2k_1+2k_2$, and $-2k_1-2k_2$, so we will use the analysis outlined in the previous sections dealing with multiplicity six with higher coefficient linear dependence. Let $\gamma_2=1$ and $\gamma_3=2$. Observe that $\beta=\lambda-1$, $\xi=-\frac{16}{3}$, $\eta=-\frac{63764}{663}$, $\chi=-\frac{36076}{3657}$, $\omega=-\frac{314956}{38577}$. 
    Let $\lambda=\lambda_0=1$ and consider the straight line orbits of the system. From Theorem \ref{tm7.1}, we see that $max(A)=\xi=-\frac{16}{3}$, and thus the transition is Type I because $\gamma_3>(\text{max} A)\gamma_2^2$. 
    Now let $\lambda=1.1$ so that we may consider the pattern formation that results from the dynamic transition as $\lambda$ crossed the critical threshold. The trivial solution $(r_1,r_2,r_3)=(0,0,0)$
 obviously becomes unstable as $\lambda>\lambda_0=1$. 
    Next consider the solution $(r_1,r_2,r_2)=(\sqrt{\frac{-\beta}{\xi}},0,0)=(\sqrt{\frac{3}{160}},0,0)$. The Jacobian evaluated at this solution is 
    \begin{align}J=\begin{pmatrix}-0.2&0&0\\0&-1.703&0\\0&0&-0.085
    \end{pmatrix},\end{align}
    and so the solution is stable. Observe that because $\beta$, $\xi$, $\eta$, and $\chi$ are all negative, the solutions $(r_1,r_2,r_3)=(0,\sqrt{\frac{-\beta}{\xi}},0)$ and $(r_1,r_2,r_3)=(0,0,\sqrt{\frac{-\beta}{\xi}})$ are also both stable as their Jacobians are diagonal matrices with negative entries.
   The solution $(r_1,r_2,r_2)=(\sqrt{\frac{-\beta}{\xi}},0,0)$ can be written as
   \begin{align}u(x,t)=2(a_1\cos(k_1^c\cd x)-a_2\sin(k_1^c\cd x)),\end{align}
   where $a_1^2+a_2^2=r_1^2$. Graphs of this solution (and the previous two) are similar to those of the multiplicity two case. 
    Now consider the solution $(r_1,r_2,r_3)=(0, \sqrt{\frac{-\beta}{\omega+\xi}},\sqrt{\frac{-\beta}{\omega+\xi}})=(0,\frac{38577}{5207000},\frac{38577}{5207000}) $. The Jacobian evaluated at this solution is
    \begin{align}J=\begin{pmatrix}0.094&0&0\\0&0.0987&-0.000896\\0&-0.000896&0.987
    \end{pmatrix},\end{align}
    with eigenvalues $(\frac{24899}{250000},\frac{24451}{250000},\frac{47}{500})$, from which we see that the solution is unstable in all directions. 
    This solution can be written as
    \begin{align}u(x,t)=2(a_3\cos(k_2^c\cd x)-a_4\sin(k_2^c\cd x)+a_5\cos(k_3^c\cd x)-a_6\sin(k_3^c\cd x)),\end{align}
    where $a_3^2+a_4^2=r_2^2$ and $a_5^2+a_6^2=r_3^2$. Graphs of this solution are similar to those of the multiplicity four case. 
    Consider the solution $p_7$. Substituting values for $\beta$, $\xi$, $\omega$, $\eta$, and $\chi$ produce undefined values for $r_1$ and $r_2$. Subsequently, this solution does not exist for the values of the parameters chosen.

\section{Multiplicity Six with $k_3^c = k_1^c + k_2^c$}
\label{s8}

\subsection{Center manifold reduction}

Consider the same situation as above but assume that $k_3^c = k_1^c + k_2^c$. This scenario is important because different coefficients would have zero in the denominator in the previous computations. If $k_3^c = k_1^c - k_2^c$, then the lattice can be redefined with $k_2^c = -k_2^c$ and the rest of these computations follow. In this case, \begin{align}
    S=\{(n_1^c,n_2^c),(-n_1^c,-n_2^c),(n_3^c,n_4^c),(-n_3^c,-n_4^c),(n_5^c,n_6^c),(-n_5^c,-n_6^c)\}.
\end{align} 
Thus, $\lambda_0 = |k_1^c|^2$,  $\beta_{n_1n_2}(\lambda_0)<0$ for all $(n_1,n_2)\in \Z ^2 \setminus (S\cup \{(0,0)\})$, and  the eigenvalue 
\begin{align*}\beta_{n_1^cn_2^c}(\lambda) = -|n_1^c k_1 +n_2^c k_2|^2 (|n_1^c k_1 +n_2^c k_2|^2 -\lambda) = -|k_c|^2(|k_c|^2-\lambda),\end{align*} 
has multiplicity six with 
\begin{align*} e_1 &= e^{i(k_1^c \cd x)},&e_2 &= e^{-i(k_1^c \cd x)},&e_3 &= e^{i(k_2^c \cd x)},\\e_4 &= e^{-i(k_2^c \cd x)},&e_5 &= e^{i(k_3^c \cd x)},&e_6 &= e^{-i(k_3^c \cd x)},\\
E_1^\lambda &= \text{span}\{e_1,...,e_6\}, & E_2^\lambda &= \overline{ \text{span} \{e_7,e_8,...\} }.
\end{align*}
The solution can thus be written as 
\begin{align} u(x,t)=\sum_{i=1}^6y_ie_i+z, \qquad z\in E_2^\lambda.
\end{align} 
 By similar computation, the center manifold function up to higher order terms is given by $\phi(x) = \sum_{i=7}^{18} \phi_i e_i$, where  \begin{align*}
e_7&= e^{2ik_1^c\cd x},&e_8&= e^{-2ik_1^c\cd x},&e_9&= e^{2ik_2^c\cd x},\\
e_{10}&= e^{-2ik_2^c\cd x},&e_{11}&= e^{2ik_3^c\cd x},&e_{12}&= e^{-2ik_3^c\cd x},\\
e_{13}&= e^{i(k_1^c-k_2^c)\cd x},&e_{14}&= e^{-i(k_1^c-k_2^c)\cd x},&e_{15}&= e^{i(2 k_1^c+k_2^c)\cd x},\\
e_{16}&= e^{-i(2 k_1^c+k_2^c)\cd x},&e_{17}&= e^{i(k_2^c+2 k_2^c)\cd x},&e_{18}&= e^{-i(k_2^c+2 k_2^c)\cd x}.
\end{align*} 
It can be calculated that the coefficients of the manifold are 
\begin{align*}
\phi_7&= \frac{-\gamma_2 y_1^2}{4|k_1^c|^2-\lambda},\quad &\phi_8&= \frac{-\gamma_2 y_2^2}{4|k_1^c|^2-\lambda},\\
\phi_9&= \frac{-\gamma_2 y_3^2}{4|k_1^c|^2-\lambda},\quad &\phi_{10}&= \frac{-\gamma_2 y_4^2}{4|k_1^c|^2-\lambda},\\
\phi_{11}&= \frac{-\gamma_2 y_5^2}{4|k_1^c|^2-\lambda},\quad &\phi_{12}&= \frac{-\gamma_2 y_6^2}{4|k_1^c|^2-\lambda},\\
\phi_{13}&= \frac{-2\gamma_2 y_1y_4}{|k_1^c-k_2^c|^2-\lambda},\quad &\phi_{14}&= \frac{-2\gamma_2 y_2y_3}{|k_1^c-k_2^c|^2-\lambda},\\
\phi_{15}&= \frac{-2\gamma_2 y_1y_5}{|2k_1^c+k_2^c|^2-\lambda},\quad &\phi_{16}&= \frac{-2\gamma_2 y_2y_6}{|2k_1^c+k_2^c|^2-\lambda},\\
\phi_{17}&= \frac{-2\gamma_2 y_3y_5}{|k_1^c+2k_2^c|^2-\lambda},\quad &\phi_{18}&= \frac{-2\gamma_2 y_4y_6}{|k_1^c+2k_2^c|^2-\lambda}.
\end{align*} 
Let
$$
y_1=y_2^\ast = a_1 + a_2 i,
\quad y_3 = y^\ast_4 = a_3 + a_4 i,
\quad y_5= y_6^\ast= a_5 + a_6 i,
$$
and 
\begin{align*}
\xi &= 2|k_1^c|^2 (\frac{2\gamma_2 ^2}{|k_1^c-k_2^c|^2 - |k_1^c|^2}-3\gamma_3),\\
\eta &=-3|k_1^c|^2 \gamma_3 + \frac{2}{3}\gamma_2^2\\
\chi &= 2|k_1^c|^2 (\frac{2\gamma_2 ^2}{|2k_1^c+k_2^c|^2 - |k_1^c|^2}-3\gamma_3),\\
\omega &= 2|k_1^c|^2 (\frac{2\gamma_2 ^2}{|k_1^c+2k_2^c|^2 - |k_1^c|^2}-3\gamma_3),\\
\tau &= 2|k_1^c|^2\gamma_2.
\end{align*} 

Up to the higher-order $o(3)$ terms,  the reduced system is given by 
\begin{equation}\label{8.3}
\begin{aligned}
    a_{1t}= \beta a_1 +a_1 (\eta (a_1^2 + a_2^2) +\xi (a_3^2 + a_4^2)+\chi (a_5^2 + a_6^2))-\tau a_3a_5 -\tau a_4 a_6,\\
    a_{2t}= \beta a_2 +a_2 (\eta (a_1^2 + a_2^2) +\xi (a_3^2 + a_4^2)+\chi (a_5^2 + a_6^2))-\tau a_3a_6 +\tau a_4 a_5,\\
    a_{3t}= \beta a_3 + a_3 (\xi (a_1^2 + a_2^2) +\eta (a_3^2 + a_4^2)+\omega (a_5^2 + a_6^2))-\tau a_1a_5 -\tau a_2 a_6,\\
    a_{4t}= \beta a_4 +a_4 (\xi (a_1^2 + a_2^2) +\eta (a_3^2 + a_4^2)+\omega (a_5^2 + a_6^2))-\tau a_1a_6 +\tau a_2 a_5,\\
    a_{5t}= \beta a_5 +a_5 (\chi (a_1^2 + a_2^2) +\omega (a_3^2 + a_4^2)+\eta (a_5^2 + a_6^2))-\tau a_1a_3 +\tau a_2 a_4,\\
    a_{6t}= \beta a_6 +a_6 (\chi (a_1^2 + a_2^2) +\omega (a_3^2 + a_4^2)+\eta (a_5^2 + a_6^2))-\tau a_2a_3 -\tau a_1 a_4.
\end{aligned}
\end{equation} 
From this point, the transition dynamics can be calculated using the different straight lines in a six dimensional space. In order to make the calculations simpler, we will impose another condition on the original solution of the equation: $u$ must be even in $x$.

\subsection{$u$ even in $x$}

Assume that $u(x,t)=u(-x,t)$. For this to be true, then 
\begin{align*} 
\sum\limits_{(n_1,n_2)\in \Z \times \Z^+ \setminus \{(0,0)\}}(z_{n_1n_2}(t)e^{i(n_1(k_1 \cd x)+n_2(k_2 \cd x))}+\conj{z_{n_1n_2}(t)}e^{-i(n_1(k_1 \cd x)+n_2(k_2 \cd x))}) \\
= \sum\limits_{(n_1,n_2)\in \Z \times \Z^+ \setminus \{(0,0)\}}(z_{n_1n_2}(t)e^{-i(n_1(k_1 \cd x)+n_2(k_2 \cd x))}+\conj{z_{n_1n_2}(t)}e^{i(n_1(k_1 \cd x)+n_2(k_2 \cd x))}). 
\end{align*} 
This condition implies that the expansion coefficients $z(t) = \conj {z(t)}$, so $z(t)\in \R$ for all $t$. 
This also means that 
\begin{align*} u(x,t) &= \sum\limits_{(n_1,n_2)\in \Z \times \Z^+ \setminus \{(0,0)\}}z_{n_1n_2}(t)(e^{i(n_1(k_1 \cd x)+n_2(k_2 \cd x))}+e^{-i(n_1(k_1 \cd x)+n_2(k_2 \cd x))}) \\&= \sum\limits_{(n_1,n_2)\in \Z \times \Z^+ \setminus \{(0,0)\}}(2z_{n_1n_2}(t)\cos(n_1(k_1 \cd x)+n_2(k_2 \cd x))\\&= \sum\limits_{(n_1,n_2)\in \Z \times \Z^+ \setminus \{(0,0)\}}(\tilde{z}_{n_1n_2}(t)\cos(n_1(k_1 \cd x)+n_2(k_2 \cd x)).
\end{align*} 
For the remaining part of this paper, we will be suppressing the tilde.\\

Now assume that $\# S = 3$ (this can occur since $u$ is even so the negative of a mode is the same as the mode itself) and $k_3^c = k_1^c + k_2^c$. In this case, 
\begin{align}
    S=\{(n_1^c,n_2^c),(n_3^c,n_4^c),(n_5^c,n_6^c)\}.
\end{align} 
Thus, we have $\beta_{n_1n_2}(\lambda_0)<0$ for all $(n_1,n_2)\in \Z \times \Z^+ \setminus (S\cup \{(0,0)\})$. Thus, with critical value $\lambda_0$, the eigenvalue 
\begin{align*}
\beta_{n_1^cn_2^c}(\lambda) = -|n_1^c k_1 +n_2^c k_2|^2 (|n_1^c k_1 +n_2^c k_2|^2 -\lambda) = -|k_c|^2(|k_c|^2-\lambda),\end{align*} 
has multiplicity six with 
\begin{equation}
\begin{aligned} 
& e_1 = \cos(k_1^c \cd x),&& e_2 = \cos(k_2^c \cd x),&&e_3 = \cos(k_3^c \cd x),\\
& E_1^\lambda = \text{span}\{e_1,e_2,e_3\},  && E_2^\lambda = \overline{ \text{span} \{e_4,e_5,...\} }.
\end{aligned}
\end{equation}
The solution can thus be written as 
\begin{align} 
u(x,t)=\sum_{i=1}^3y_ie_i+z, \quad z\in E_2^\lambda. \label{7.14}
\end{align} 
By similar computation, the center manifold function up to higher order terms is given by $\phi(x) = \sum_{i=4}^{9} \phi_i e_i$. Using the notation that \begin{equation}\begin{aligned}
e_4&= \cos(2k_1^c \cd x),&e_5&= \cos(2k_2^c \cd x),\\e_6&= \cos(2k_3^c \cd x),&
e_7&= \cos((k_1^c-k_2^c) \cd x),\\e_8&= \cos((2k_1^c+k_2^c) \cd x),&e_9&= \cos((k_1^c+2k_2^c) \cd x),
\end{aligned}\end{equation} it can be calculated that the coefficients of the manifold are \begin{equation}\begin{aligned}
\phi_4&= \frac{-\gamma_2 y_1^2}{8|k_1^c|^2-2\lambda},\quad &\phi_5&= \frac{-\gamma_2 y_2^2}{8|k_1^c|^2-2\lambda},\\
\phi_6&= \frac{-\gamma_2 y_3^2}{8|k_1^c|^2-2\lambda},\quad &\phi_7&= \frac{-\gamma_2 y_1y_2}{|k_1^c - k_2^c|^2-\lambda},\\
\phi_8&= \frac{-\gamma_2 y_1y_3}{|2k_1^c + k_2^c|^2-\lambda},\quad &\phi_9&= \frac{-\gamma_2 y_2y_3}{|k_1^c + 2k_2^c|^2-\lambda}.
\end{aligned}\end{equation} Using this function and by letting \begin{equation}\begin{aligned}\lambda &= \lambda_0 = |k_1^c|^2,\\
\xi &= -\frac{3}{4}|k_1^c|^2 \gamma_3 + \frac{1}{6}\gamma_2^2,\\
\eta &=-\frac{3}{2}|k_1^c|^2 \gamma_3 + \frac{|k_1^c|^2}{|k_1^c - k_2^c|^2 - |k_1^c|^2}\gamma_2^2, \\
\chi &= -\frac{3}{2}|k_1^c|^2 \gamma_3 + \frac{|k_1^c|^2}{|2 k_1^c + k_2^c|^2 - |k_1^c|^2}\gamma_2^2 ,\\
\omega &= -\frac{3}{2}|k_1^c|^2 \gamma_3 + \frac{|k_1^c|^2}{|k_1^c + 2 k_2^c|^2 - |k_1^c|^2}\gamma_2^2 ,\\
\tau &= -|k_1^c|^2\gamma_2,\end{aligned}
\end{equation} 
the reduced system can be rewritten as 
\begin{equation}\label{8.10}
\begin{aligned}
    y_{1t}= \beta y_1 +y_1 (\xi y_1^2 + \eta y_2^2 + \chi y_3^2)+\tau y_2y_3 +o(3),\\
    y_{2t}= \beta y_2 +y_2 (\eta y_1^2 + \xi y_2^2 + \omega y_3^2)+\tau y_1y_3 +o(3),\\
    y_{3t}= \beta y_3 +y_3 (\chi y_1^2 + \omega y_2^2 + \xi y_3^2)+\tau y_1y_2 +o(3).
\end{aligned}
\end{equation} 
By algebraic calculations, it can be shown that if $|k_1^c|^2 = |k_2^c|^2 = |k_3^c|^2$ and $k_3^c = k_1^c + k_2^c$, then $|k_1^c - k_2^c|^2 = |2k_1^c + k_2^c|^2 = |k_1^c + 2 k_2^c|^2$. This implies that 
\begin{equation}
\begin{aligned}
  &  \xi = -\frac{3}{4}|k_1^c|^2 \gamma_3 + \frac{1}{6}\gamma_2^2,\\
 &   \eta = \chi = \omega =-\frac{3}{2}|k_1^c|^2 \gamma_3 + \frac{|k_1^c|^2}{|k_1^c - k_2^c|^2 - |k_1^c|^2}\gamma_2^2,\\
  &  \tau = -|k_1^c|^2\gamma_2,
\end{aligned}
\end{equation} 
and the reduced system can be rewritten as 
\begin{equation}\label{8.12}
\begin{aligned}
    y_{1t}= \beta y_1 + \xi y_1^3 +\eta y_1 y_2^2 + \eta y_1y_3^2+\tau y_2y_3 +o(3),\\
    y_{2t}= \beta y_2 +\xi y_2^3 +\eta y_2 y_3^2 + \eta y_1y_2^2+\tau y_1y_3 +o(3),\\
    y_{3t}= \beta y_3 +\xi y_3^3 +\eta y_1^2 y_3 + \eta y_2^2y_3+\tau y_1y_2 +o(3).
\end{aligned}
\end{equation}

\begin{theorem}[Transition Types with $k_3^c = k_1^c + k_2^c$]\label{tm8.1}
Consider the Cahn-Hilliard system (\ref{2.5}) with evenness assumption and assume that $k_3^c = k_1^c + k_2^c$. Then the following assertions hold true.
\begin{itemize}
\item[1).] If $\gamma_2 = 0$, then as $\lambda$ crosses $\lambda_0$, the system undergoes a continuous (Type I) transition to $\Sigma_\lambda\approx S^2$ if $\gamma_3 >0$, and undergoes a jump (Type II) transition if $\gamma_3 <0$.

\item[2).] If $\gamma_2 > 0$, then as $\lambda$ crosses $\lambda_0$, the system undergoes a  jump (Type II) transition if $\gamma_3 < \frac{2}{9|k_1^c|^2}\gamma_2^2$, and a continuous (Type I) transition to $\Sigma_\lambda\approx S^2$ if $\gamma_3 > \frac{2}{9|k_1^c|^2}\gamma_2^2$.

\item[3).] If $\gamma_2 < 0$,  then as $\lambda$ crosses $\lambda_0$, the system undergoes a jump (Type II) transition.
 \end{itemize}
  
\end{theorem} 

\begin{proof}
{\sc Case $\gamma_2 = 0$.} Under the assumption, we have  \begin{align}
        \xi = -\frac{3}{4}|k_1^c|^2 \gamma_3,\quad  \eta =-\frac{3}{2}|k_1^c|^2 \gamma_3,  \quad \tau = 0.
    \end{align} 
    The reduced  system then becomes \begin{align}
    y_{1t}= \beta y_1 + \xi y_1 ( y_1^2 + 2 y_2^2 + 2 y_3^2)+o(3),\\
    y_{2t}= \beta y_2 +\xi y_2 (2 y_1^2 + y_2^2 + 2 y_3^2)+o(3),\\
    y_{3t}= \beta y_3 +\xi y_3 (2 y_1^2 + 2 y_2^2 +  y_3^2)+o(3).
\end{align} 
If $\xi >0$, then all solutions will tend away from the origin, and if $\xi <0$ all solutions will tend towards the origin. Equivalently, if $\gamma_3 <0$, the all solutions will tend away from the origin, and if $\gamma_3 >0$ all solutions will tend towards the origin. Therefore, if $\gamma_3<0$, the transition is Type II and if $\gamma_3>0$, the transition is Type I.

\medskip

{\sc Case $\gamma_2 > 0$.} In this case,  $\tau <0$. Then straight lines  to the reduced  system are 
\begin{align*}
 &   y_1=y_2=0,\\
  &  y_1=y_3=0,\\
&    y_2=y_3=0,\\
 &   y_1=0 \text{ and } y_2^2=y_3^2,\\
 &   y_2=0 \text{ and } y_1^2=y_3^2,\\
 &   y_3=0 \text{ and } y_1^2=y_2^2,\\
 &   y_1^2=y_2^2=y_3^2.
\end{align*} 
Let $i,j,k\in \{1,2,3\}$ such that $i\neq j$, $i\neq k$, and $j\neq k$. Along the lines of the form $y_i = y_j = 0$, the system reduces to \begin{align}
    y_{kt} = \xi y_k^3.
\end{align} Observe that since $\xi = -\frac{3}{4}|k_1^c|^2 \gamma_3 + \frac{1}{6}\gamma_2$, \begin{align}
    \xi >0 \iff \gamma_3 < \frac{2}{9|k_1^c|^2}\gamma_2^2.
\end{align}It can be seen that if $\xi <0$, solutions along these lines tend towards the origin and if $\xi >0$, solutions along these lines tend away from the origin. Equivalently, if $\gamma_3 > \frac{2}{9|k_1^c|^2}\gamma_2^2$, solutions along these lines tend towards the origin and if $\gamma_3 < \frac{2}{9|k_1^c|^2}\gamma_2^2$, solutions along these lines tend away from the origin. It can also be seen that along the lines of the form $y_k = 0$ and $y_i^2 = y_j^2$, there are no straight line solutions because at least one of $y_i$ and $y_j$ must be zero.\\

Along the line $y_1=y_2=y_3$ the system reduces to 
$$
    y_{1t}= \beta y_1 +y_1^3(\xi +2 \eta )+\tau y_1^2+o(3), \qquad y_1=y_2=y_3.$$
     By truncating this system to second order (which can be done since we are considering small perturbations near the origin), the system becomes $$
    y_{1t}= \tau y_1^2, \qquad y_1=y_2=y_3. $$
Since $\tau <0$, all solutions along this straight line that start near zero tend towards zero.

\medskip

{\sc Case 3 $\gamma_2 < 0$.} In this case, we have  $\tau >0$. The straight lines corresponding to this system are 
\begin{align*}
  &  y_1=y_2=0,\\
   & y_1=y_3=0,\\
  &  y_2=y_3=0,\\
  &  y_1=0 \text{ and } y_2^2=y_3^2,\\
  &  y_2=0 \text{ and } y_1^2=y_3^2,\\
   & y_3=0 \text{ and } y_1^2=y_2^2,\\
   & y_1^2=y_2^2=y_3^2.
\end{align*} 
Let $i,j,k\in \{1,2,3\}$ such that $i\neq j$, $i\neq k$, and $j\neq k$. Along the lines of the form $y_i = y_j = 0$, the system reduces to \begin{align}
    y_{kt} = \xi y_k^3.
\end{align} Observe that since $\xi = -\frac{3}{4}|k_1^c|^2 \gamma_3 + \frac{1}{6}\gamma_2$, \begin{align}
    \xi >0 \iff \gamma_3 < \frac{2}{9|k_1^c|^2}\gamma_2^2.
\end{align}It can be seen that if $\xi <0$, solutions along these lines tend towards the origin and if $\xi >0$, solutions along these lines tend away from the origin. Equivalently, if $\gamma_3 > \frac{2}{9|k_1^c|^2}\gamma_2^2$, solutions along these lines tend towards the origin and if $\gamma_3 < \frac{2}{9|k_1^c|^2}\gamma_2^2$, solutions along these lines tend away from the origin. It can also be seen that along the lines of the form $y_k = 0$ and $y_i^2 = y_j^2$, there are no straight line solutions because at least one of $y_i$ and $y_j$ must be zero.\\

Along the line $y_1=y_2=y_3$ the system reduces to 
$$
    y_{1t}= \beta y_1 +y_1^3(\xi +2 \eta )+\tau y_1^2 +o(3),\qquad y_1=y_2=y_3.$$
As before, the solution is dictated by the positive sign of $\tau >0$, and  all solutions along this straight line that start near zero tend away from zero.
    
    The proof is then complete.
\end{proof}

\subsection{Structure of the set of transition states}\label{s8.3}
Consider the reduced system (\ref{8.12}).
Assume that $\gamma_2 <0$ so that $\tau >0$ (if $\gamma_2 >0$, the same equilibria and stability will persist). The nontrivial equilibria of this system can be calculated as in Hoyle \cite{hoy} to be 
\begin{itemize}
    \item \text{Rolls: } 
    \begin{itemize}
    \item $(y_1,y_2,y_3)=(\pm \sqrt{\frac{-\beta}{\xi}},0,0)$, \item $(y_1,y_2,y_3)=(0,\pm \sqrt{\frac{-\beta}{\xi}},0)$, 
    \item $(y_1,y_2,y_3)=(0,0,\pm \sqrt{\frac{-\beta}{\xi}})$;
    \end{itemize}
 
 \item \text{Hexagons: } 
 	\begin{itemize} 
	\item $y_1=y_2=y_3=\frac{-\tau \pm \sqrt{\tau ^2 -4\beta \xi - 8\beta \eta}}{2\xi + 4\eta} \text{ if $\tau ^2 -4\beta \xi - 8\beta \eta \geq 0$}$;			\end{itemize}
  
  \item \text{Rectangles: } 
  	\begin{itemize}
	\item $(y_1,y_2,y_3) = (\frac{\tau}{\xi - \eta},\pm\sqrt{\frac{-1}{\xi + \eta}(\beta +\frac{\tau ^2 \xi}{(\xi-\eta)^2})},\pm\sqrt{\frac{-1}{\xi + \eta}(\beta +\frac{\tau ^2 \xi}{(\xi-\eta)^2})})$, 
	\item $(y_1,y_2,y_3) = (\pm\sqrt{\frac{-1}{\xi + \eta}(\beta +\frac{\tau ^2 \xi}{(\xi-\eta)^2})},\frac{\tau}{\xi - \eta},\pm\sqrt{\frac{-1}{\xi + \eta}(\beta +\frac{\tau ^2 \xi}{(\xi-\eta)^2})})$,
    \item $(y_1,y_2,y_3) = (\pm\sqrt{\frac{-1}{\xi + \eta}(\beta +\frac{\tau ^2 \xi}{(\xi-\eta)^2})},\pm\sqrt{\frac{-1}{\xi + \eta}(\beta +\frac{\tau ^2 \xi}{(\xi-\eta)^2})},\frac{\tau}{\xi - \eta})$.
    \end{itemize}
\end{itemize}

Some of these solutions are in far-fields, so we will not consider those. The solutions that are not in far fields are \begin{itemize}
    \item \text{Rolls: } \begin{itemize}\item $\pm p_1 =(\pm \sqrt{\frac{-\beta}{\xi}},0,0)$, \item $\pm p_2 =(0,\pm \sqrt{\frac{-\beta}{\xi}},0)$, \item $\pm p_3 =(0,0,\pm \sqrt{\frac{-\beta}{\xi}})$;\end{itemize}
    \item \text{Hexagons (if $\tau ^2 -4\beta \xi - 8\beta \eta \geq 0$): } \begin{itemize} \item $p_4=(\frac{-\tau + \sqrt{\tau ^2 -4\beta \xi - 8\beta \eta}}{2\xi + 4\eta},\frac{-\tau + \sqrt{\tau ^2 -4\beta \xi - 8\beta \eta}}{2\xi + 4\eta},\frac{-\tau + \sqrt{\tau ^2 -4\beta \xi - 8\beta \eta}}{2\xi + 4\eta})$.\end{itemize}
\end{itemize} The Jacobian of this system at a fixed point $(y_1,y_2,y_3)$ is 
$$
    J=\begin{pmatrix} \beta + 3\xi y_1^2 + \eta y_2^2 + \eta y_3^2 & 2\eta y_1y_2 + \tau y_3 & 2\eta y_1y_3 + \tau y_2 \\ 2\eta y_1 y_2 +\tau y_3 & \beta + \eta y_1^2 + 3\xi y_2^2 + \eta y_3^2 & 2\eta y_2y_3 + \tau y_1\\ 2\eta y_1 y_3 + \tau y_2 & 2\eta y_2y_3 + \tau y_1 & \beta + \eta y_1^2 + \eta y_2^2 + 3\xi y_3^2   \end{pmatrix}.
$$ 
The stability of each solution can be determined from calculating the eigenvalues of this matrix at each equilibrium.\\

\begin{theorem}[Stability of Roll Solutions]\label{tm8.2}
 Consider the roll solutions $\pm p_1$, $\pm p_2$, and $\pm p_3$ and assume $\beta >0$ (or equivalently $\lambda > \lambda_0$).

\begin{itemize}
\item[1)]  If $\eta < \xi - \tau \sqrt{\frac{-\xi}{\beta}}$, then the roll solutions all have three stable eigenvalues. Namely, all rolls are local attractors.

\item[2)] If $\xi - \tau \sqrt{\frac{-\xi}{\beta}}<\eta < \xi + \tau \sqrt{\frac{-\xi}{\beta}}$, then the roll solutions all have two stable eigenvalues and one unstable eigenvalue. Moreover, the unstable directions for $\pm p_1$ are $(0,\pm 1,1)$, for $\pm p_2$ are $(\pm 1,0,1)$, and for $\pm p_3$ are $(\pm 1,1,0)$. Namely all rolls are saddles with two-dimensional stable manifolds and one-dimensional unstable manifolds.

\item[3)] If $\eta > \xi + \tau \sqrt{\frac{-\xi}{\beta}}$, then the roll solutions all have one stable eigenvalue and two unstable eigenvalue. Namely, all  rolls are saddles with one-dimensional stable manifolds and two-dimensional unstable manifolds.

  \end{itemize}
  
\end{theorem}

\begin{proof}
Consider the roll solutions $\pm p_1$. It can be seen that \begin{align}
    J(\pm p_1)=\begin{pmatrix} \beta + 3\xi |\frac{-\beta}{\xi}|  & 0 & 0 \\ 0 & \beta + \eta |\frac{-\beta}{\xi}|  &  \pm \tau \sqrt{\frac{-\beta}{\xi}}\\ 0 &  \pm \tau \sqrt{\frac{-\beta}{\xi}} & \beta + \eta |\frac{-\beta}{\xi}|    \end{pmatrix},\\
    J(\pm p_2)=\begin{pmatrix} \beta + \eta |\frac{-\beta}{\xi}|  & 0 & \pm \tau \sqrt{\frac{-\beta}{\xi}} \\ 0 & \beta + 3\xi |\frac{-\beta}{\xi}|  &  0\\ \pm \tau \sqrt{\frac{-\beta}{\xi}} &  0 & \beta + \eta |\frac{-\beta}{\xi}|    \end{pmatrix},\\
    J(\pm p_3)=\begin{pmatrix} \beta + \eta |\frac{-\beta}{\xi}|  & \pm \tau \sqrt{\frac{-\beta}{\xi}} & 0 \\ \pm \tau \sqrt{\frac{-\beta}{\xi}} & \beta + \eta |\frac{-\beta}{\xi}|  &  0\\ 0 &  0 & \beta + 3\xi |\frac{-\beta}{\xi}|    \end{pmatrix}.
\end{align} 
By direct computation, the eigenvalues and eigenvectors of $J(\pm p_1)$ are \begin{equation}\begin{aligned}\lambda_1 &= \beta + 3\xi |\frac{-\beta}{\xi}|,&& v_1 = (1,0,0),\\
\lambda_2 &= \beta + \eta |\frac{-\beta}{\xi}| + \tau \sqrt{\frac{-\beta}{\xi}},&& v_2 = (0,\pm 1,1),\\ 
\lambda_3 &= \beta + \eta |\frac{-\beta}{\xi}| - \tau \sqrt{\frac{-\beta}{\xi}},&& v_3 = (0,\mp 1,1).
\end{aligned}
\end{equation}

The eigenvalues and eigenvectors of $J(\pm p_2)$ are \begin{equation}\begin{aligned}\lambda_1 &= \beta + 3\xi |\frac{-\beta}{\xi}|,&& v_1 = (1,0,0),\\
\lambda_2 &= \beta + \eta |\frac{-\beta}{\xi}| + \tau \sqrt{\frac{-\beta}{\xi}},&& v_2 = (\pm 1,0,1),\\ 
\lambda_3 &= \beta + \eta |\frac{-\beta}{\xi}| - \tau \sqrt{\frac{-\beta}{\xi}},&& v_3 = (\mp 1,0,1).\end{aligned}
\end{equation}
The eigenvalues and eigenvectors of $J(\pm p_3)$ are 
\begin{equation}\begin{aligned}\lambda_1 &= \beta + 3\xi |\frac{-\beta}{\xi}|,&& v_1 = (1,0,0),\\
\lambda_2 &= \beta + \eta |\frac{-\beta}{\xi}| + \tau \sqrt{\frac{-\beta}{\xi}},&&v_2 = (\pm 1,1,0),\\ 
\lambda_3 &= \beta + \eta |\frac{-\beta}{\xi}| - \tau \sqrt{\frac{-\beta}{\xi}},&&v_3 = (\mp 1,1,0).\end{aligned}\end{equation}
Now assume that $\beta >0$, so $\xi <0$ must be true. In this case, the eigenvalues reduce to \begin{equation}\begin{aligned}\lambda_1 &= -2 \beta ,\\
\lambda_2 &= \beta + \eta \frac{-\beta}{\xi} + \tau \sqrt{\frac{-\beta}{\xi}},\\ 
\lambda_3 &= \beta + \eta \frac{-\beta}{\xi} - \tau \sqrt{\frac{-\beta}{\xi}}.\end{aligned}\end{equation} From this, the following statements emerge: \begin{equation}\begin{aligned}\lambda_1 &<0 ,\\
\lambda_2 &>0 \iff \eta > \xi - \tau \sqrt{\frac{-\xi}{\beta}},\\ 
\lambda_3 &>0 \iff \eta > \xi + \tau \sqrt{\frac{-\xi}{\beta}}.\end{aligned}\end{equation} Since $\xi <0$, if $\eta < \xi - \tau \sqrt{\frac{-\xi}{\beta}}$, these solutions will have three stable eigenvalues, if $\xi - \tau \sqrt{\frac{-\xi}{\beta}}<\eta < \xi + \tau \sqrt{\frac{-\xi}{\beta}}$, the solution will have two stable eigenvalues and one unstable eigenvalue, and if $\eta > \xi + \tau \sqrt{\frac{-\xi}{\beta}}$, these solutions will have two unstable eigenvalues and one stable eigenvalue.\end{proof}

Consider the hexagon solution $p_4$. It can be seen that \begin{align}
    J(p_4)=\begin{pmatrix} \beta + 3\xi y_1^2 + 2\eta y_1^2  & 2\eta y_1^2 + \tau y_1 & 2\eta y_1^2 + \tau y_1 \\ 2\eta y_1^2 +\tau y_1 & \beta + 2\eta y_1^2 + 3\xi y_1^2  & 2\eta y_1^2 + \tau y_1\\ 2\eta y_1^2y_1^2 + \tau y_1 & 2\eta y_1^2 + \tau y_1 & \beta + 2\eta y_1^2 + 3\xi y_1^2   \end{pmatrix}.
\end{align} This solution will be further explored in the example below.

\subsection{Example: roll and hexagonal patterns}\label{s8.4}
\-\quad\ Let $l_1=(\frac{\pi}{25},-\frac{\sqrt{3}\pi}{75})$ and $l_2=(0,-\frac{2\sqrt{3}\pi}{75})$. The dual lattice is spanned by the vectors $k_1 = (50,0)$ and $k_2 = (-25,-25\sqrt{3})$. Note that $k_1 + k_2=(25,-25\sqrt{3})$ and $|k_1|^2 = |k_2|^2 = |k_1+k_2|^2=2500$. The critical points of the lattice are thus $k_1$, $-k_1$, $k_2$, $-k_2$, $k_1+k_2$, and $-k_1-k_2$, so we will use the analysis outlined in the previous sections dealing with multiplicity six where $k_1^c = k_1$, $k_2^c = k_2$, and $k_3^c = k_1 + k_2$. Observe that \(\beta=2500\lambda-6250000\), \(\xi=-1875\gamma_3+\frac{1}{6}\gamma_2^2\), \(\eta=-3750 \gamma_3+\frac{1}{2}\gamma_2^2\), and $\tau = -2500\gamma_2$. \\

Let $\lambda=\lambda_0=1$ and consider the straight line orbits of the system. Assume $\gamma_2=0$. From Theorem \ref{tm8.1}, we see that the transition is Type I if $\gamma_3>0$ and Type II if $\gamma_3<0$. For $\gamma_2>0$, the transition is Type I if $\gamma_3>\frac{2}{9}\gamma_2^2$ and Type II if $\gamma_3<\frac{2}{9}\gamma_2^2$. For $\gamma_2<0$, the transition is always Type II and solutions along these orbits will tend away from the origin.\\

Now let $\lambda=2501$ so that we may consider the pattern formation that results from the dynamic transition at $\lambda=\lambda_0=2500$. Consider the trivial solution $y_1=y_2=y_3=0$. This solution loses its stability as the control parameter exceeds the critical threshold, i.e. when $\lambda>\lambda_0=2500$. Now consider the rolls solutions 
\begin{align} 
\label{8.30}
y_i=\pm \sqrt{\frac{37500000-15000\lambda}{-11250\gamma_3+\gamma_2^2}} \text{, $y_j=0$, $y_k=0$, $i \neq j \neq k$} ,
\end{align} 
for $i \in \{1,2,3\}$. 
Theorem \ref{tm8.2} determine the stability of these six solutions in terms of $\gamma_2$ and $\gamma_3$. Let $\gamma_2=1$, $\gamma_3=2$, and $\lambda=2501$. Then $\xi=-\frac{22499}{6}$, $\eta=-\frac{14999}{2}$, $\tau=-2500$, and $\beta=1$. From this, we can calculate that $\xi+\tau \sqrt{\frac{-\xi}{\beta}}=-156839$ and $\xi-\tau\sqrt{\frac{-\xi}{\beta}}=149340$. It then follows that 
\begin{align}\xi+\tau \sqrt{\frac{-\xi}{\beta}}<\eta<\xi-\tau \sqrt{\frac{-\xi}{\beta}},\end{align}
and assertion 2) of the theorem applies. We see that the rolls each have two stable eigenvalues and one unstable eigenvalue. The unstable directions for $p_1$ are $(0,\pm 1, 1)$, for $p_2$ are $(\pm 1, 0, 1)$, and for $p_3$ are $(\pm 1, 1, 0)$. 
From (\ref{7.14}), by ignoring higher order terms, we can write the rolls solutions as
\begin{align}u(x,t)=\sum_{i=1}^{3}y_ie_i,\end{align} 
where the coefficients $y_i$ are found by means of equation (\ref{8.30}); in this case our nontrivial coefficient is $\pm0.817$ and the others are zero . Thus, our six solutions are
\begin{equation}\begin{aligned}u_{1,2}(x,t)&=\pm0.817\cos(50x_1),\\u_{3,4}(x,t)&=\pm0.817\cos(-25x_1-25\sqrt{3}x_2),\\u_{5,6}(x,t)&=\pm0.817\cos(25x_1-25\sqrt{3}x_2).\end{aligned}\end{equation}

Figure \ref{f8.1} shows a graph of the solution $u_1$. Notice that the rolls are vertical, a result of the $x_1$ term inside the cosine. In contrast, Figure \ref{f8.2} shows a graph of the solution $u_3$ where the rolls are oriented at a different angle. 
\begin{figure}
    \centering
    \includegraphics[width=0.5\linewidth]{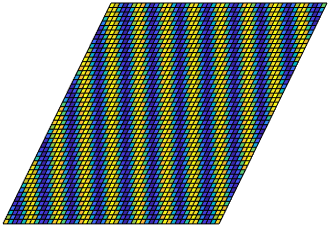}
    \caption{Rolls exhibited by the stationary solution $u(x,t)=0.817\cos(50x_1)$}
    \label{f8.1}
\end{figure}
\begin{figure}
    \centering
    \includegraphics[width=0.5\linewidth]{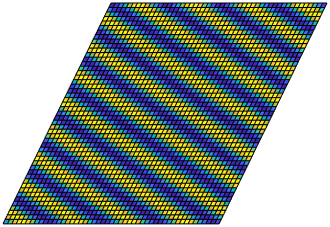}
    \caption{Rolls exhibited by the stationary solution $u(x,t)=0.817\cos(-25x_1-25\sqrt{3}x_2)$}
    \label{f8.2}
\end{figure}

Consider the hexagonal solution $p_4$ given by
\begin{align} y_1=y_2=y_3=\frac{-\tau+\sqrt{\tau^2-4\beta \xi-8\beta \eta}}{2\xi+4\eta}. \end{align}
Using our previous values of $\gamma_2$, $\gamma_3$, and $\lambda$, the solution becomes
\begin{align}y_1=y_2=y_3=-0.134.\end{align}
We look now at the Jacobian (15.37) evaluated at this solution. Plugging in coefficients, we get
\begin{align}J(p_4)=\begin{pmatrix} -470&65.68&65.68\\65.68&-470&65.68\\65.68&65.68&-470 \end{pmatrix} .\end{align}

The eigenvalues of this matrix are $(-\frac{13392}{25},-\frac{13392}{25}, -\frac{-8466}{25})$, all of which are stable. Thus, the stationary solution is stable and can be written as
\begin{align}u_7(x,t)=
-0.134[\cos(50x_1)+\cos(-25x_1-25\sqrt{3}x_2)+\cos(25x_1-25\sqrt{3}x_2)].\end{align}
Figure \ref{f8.3} shows a graph of this solution. Notice that the circles are hexagonally-packed, in contrast with the square-packed circles of Figure \ref{f5.4}. In fact, hexagonally-packed circles (HPC) are not normally observed under the regular Cahn-Hilliard model on rectangular domains. It is the unique feature of our lattice structure and the high multiplicity of the critical eigenvalue that allows for this pattern to emerge. \\

Note that for the long-range interaction model in the next section, the 
 critical vector $k^c=n_1^ck_1+n_2^ck_2$ is related to the long-range interaction term in the manner 
given by (\ref{9.5}), so that as $\sigma$ gets larger, $|k^c|^2$ increases leading to richer and more complex patterns such as the hexagonal patterns; see also \cite{LSWZ11}.

\begin{figure}
    \centering
    \includegraphics[width=0.5\linewidth]{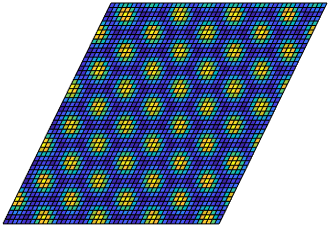}
    \caption{Hexagonally-packed circles exhibited by the stationary solution
    $u(x,t)=-0.213(\cos(50x_1)+\cos(-25x_1-25\sqrt{3}x_2)+\cos(25x_1-25\sqrt{3}x_2))$}
    \label{f8.3}
\end{figure}

\section{Long-Range Interaction}
Assume the same lattice structure as mentioned in Section \ref{s2}, but now consider any solution to the boundary value problem; see among others \cite{LSWZ11}:
 \begin{equation}\begin{aligned}
 &   u_t=-\Delta^2 u -\lambda \Delta u -\sigma u +\Delta (\gamma_2 u^2 + \gamma_3 u^3), (x,t)\in \R^2 \times \R^+,\\
 &   u(x+k,t)=u(x,t), k\in L^\ast,\\
   & u(x,0)=\phi(x),\\
 &   \int_{U}u(x,t)dx=0.
\end{aligned}\label{9.1}
\end{equation} 
In this case, the linear operator $L$ is given by $Lu=-\Delta^2u-\lambda \Delta u -\sigma u$, the eigenvalues of $L$ are 
\begin{align}
\label{9.2}\beta_{n_1n_2}(\lambda)&=-|n_1k_1+n_2k_2|^4+\lambda |n_1k_1+n_2k_2|^2 -\sigma\\&= -|n_1k_1+n_2k_2|^2 (|n_1k_1+n_2k_2|^2-\lambda )-\sigma.\nonumber
\end{align} 
To find the first mode that goes unstable, observe that if $\beta_{n_1n_2}(\lambda) = 0$, then 
$$\lambda = |n_1k_1+n_2k_2|^2 + \frac{\sigma}{|n_1k_1+n_2k_2|^2}.$$ Thus, the critical value $\lambda_0$ is given by 
\begin{align}
    \lambda_0 = \min_{k\in L^\ast \setminus \{(0,0)\}}\left(|n_1k_1+n_2k_2|^2 + \frac{\sigma}{|n_1k_1+n_2k_2|^2}\right).
\end{align} 
Let $S$  be defined by 
\begin{equation}\label{9.4}
S=\left\{(n_1,n_2)\subset \Z^2\setminus \{(0,0)\quad  \big|  \quad |n_1k_1+n_2k_2|^2 + \frac{\sigma}{|n_1k_1+n_2k_2|^2} = \lambda_0\right\}.
\end{equation}
 It can be seen that depending on parameter values, the possible values of the cardinality of $S$ are any even natural number. The critical vector $k^c=n_1^ck_1+n_2^ck_2$ is related to the long-range interaction term in the manner 
 \begin{equation}\label{9.5}
 |k^c|^2 \sim \sqrt{\sigma}.
 \end{equation} 
Consequently, as $\sigma$ gets larger, $|k^c|^2$ increases leading to richer and more complex patterns.\\

Assume going forward that $\# S=2$. Then
\begin{align}
    S=\{(n_1^c,n_2^c),(-n_1^c,-n_2^c)\}.
\end{align} 
Thus, we have $\beta_{n_1n_2}(\lambda_0)<0$ for all $(n_1,n_2)\in \Z ^2 \setminus (S\cup \{(0,0)\})$. Thus, with critical value $\lambda_0$, the eigenvalue 
\begin{align}
\beta_{n_1^cn_2^c}(\lambda) &= -|n_1^c k_1 +n_2^c k_2|^2 (|n_1^c k_1 +n_2^c k_2|^2 -\lambda)-\sigma \\&= -|k_c|^2(|k_c|^2-\lambda)-\sigma,\nonumber 
\end{align} 
has multiplicity two with 
\begin{equation}
\begin{aligned} & e_1 = e^{i(k_1^c \cd x)},  && e_2 = e^{-i(k_1^c \cd x)}\\
& E_1^\lambda =\text{span}\{e_1,e_2\}, && 
E_2^\lambda = \overline{ \text{span}\{e_3,e_4,...\} }.
\end{aligned}
\end{equation} 
The solution can thus be written as 
\begin{align} u(x,t)=y_1e_1+y_2e_2+z,
\quad z\in E_2^\lambda.
\end{align}
 By similar computation as in the previous sections, the center manifold function up to higher order terms is given by \begin{align}\phi(x) = \frac{4\gamma_2 |k_1^c|^2 y_1^2}{-4|k_1^c|^2 (4|k_1^c|^2-\lambda)-\sigma}e^{2ik_c \cd x}+\frac{4\gamma_2 |k_1^c|^2 y_2^2}{-4|k_1^c|^2 (4|k_1^c|^2-\lambda)-\sigma}e^{-2ik_c \cd x}.\end{align} Using this function, we can derive  the reduced equations, up to suppressing high-order terms, as follows: 
 \begin{equation}\begin{aligned}
    y_{1t}=\beta y_1 +\frac{8|k_1^c|^2\gamma_2^2y_1^2y_2}{-4|k_1^c|^2(4|k_1^c|^2-\lambda)-\sigma}-3|k_1^c|^2y_1^2y_2\gamma_3,\\
    y_{2t}=\beta y_2 +\frac{8|k_1^c|^2\gamma_2^2y_1y_2^2}{-4|k_1^c|^2(4|k_1^c|^2-\lambda)-\sigma}-3|k_1^c|^2y_1y_2^2\gamma_3.    
\end{aligned}
\end{equation} 
Let 
\begin{equation}
\begin{aligned}
& \lambda_0 = |k_1^c|^2 + \frac{\sigma}{|k_1^c|^2}, 
\qquad \eta = \frac{-8|k_1^c|^4}{3(-4|k_1^c|^4+\sigma)}\gamma_2^2 -3|k_1^c|^2 \gamma_3,
\end{aligned}
\end{equation} 
Let  $ y_1= a_1 + a_2 i, y_2 = a_1 - a_2 i$, then 
 the reduced system can be rewritten as 
 \begin{equation}
 \label{9.13}
 \begin{aligned}
    a_{1t}=\beta a_1 +\eta a_1 (a_1^2 + a_2^2) +o(3),\\
    a_{2t}=\beta a_2 +\eta a_2 (a_1^2 + a_2^2)+o(3).
\end{aligned}\end{equation}

 \begin{theorem}[Transition Types with Long-Range Interaction]\label{tm9.1}
    Assume the multiplicity of $\beta_1$ is two at $\lambda=|k_1^c|^2 + \frac{\sigma}{|k_1^c|^2}$. The following are true:

 \begin{itemize}
   
\item[1).]  If 
$$\gamma _3 > \frac{8|k_1^c|^2}{9(4|k_1^c|^4 - \sigma)}\gamma_2^2,$$ then the  system undergoes a continuous dynamical transition (Type I) to $\Sigma_\lambda\approx S^1$ consisting of a circle of steady-states as $\lambda$ crosses $\lambda_0$. 
\item[2).]  If  
$$\gamma _3 < \frac{8|k_1^c|^2}{9(4|k_1^c|^4 - \sigma)}\gamma_2^2,$$ then the  system undergoes a jump dynamical transition (Type II) as $\lambda$ crosses $\lambda_0$.
\end{itemize}
\end{theorem}

\begin{proof}
    By analyzing the system \begin{equation}\begin{aligned}
    a_{1t}=\eta a_1 (a_1^2 + a_2^2),\\
    a_{2t}=\eta a_2 (a_1^2 + a_2^2),
\end{aligned}\end{equation} it can be seen that all solutions tend towards the origin when $\eta <0$ and tend away from the origin when $\eta >0$. Thus, the transition is Type I when $\eta <0$ and Type II when $\eta >0$.
\end{proof}

By using the approximative system \begin{equation}\begin{aligned}
    a_{1t}=\beta a_1 + \eta a_1 (a_1^2 + a_2^2) +o(3),\\
    a_{2t}=\beta a_2 + \eta a_2 (a_1^2 + a_2^2) +o(3),
\end{aligned}\end{equation} and letting $a_1^2 + a_2^2 = r^2$, this system can be rewritten as \begin{align}
    r_t=\beta r + \eta r^3.
\end{align} The nontrivial equilibrium of this system is $r=\sqrt{\frac{-\beta}{\eta}}$. The Jacobian of this system at a fixed point $r$ is \begin{align}
    J=\begin{pmatrix} \beta +3\eta r^2  \end{pmatrix}.
\end{align} 
When $r=\sqrt{\frac{-\beta}{\eta}}$, the eigenvalue of the Jacobian is $\beta +3\eta |\frac{\beta}{\eta}|$. If $\beta >0$, then $\eta<0$ must be true, which implies that this solution will be stable.
Stationary solutions in this case are given by 
\begin{align*}
u(x,t)&=y_1e_1+\bar{y_1}\bar{e_1}\\&=(a_1+ia_2)(\cos(k_c\cdot x)+i\sin(k_c\cdot x))\\&+(a_1-ia_2)(\cos(k_c\cdot x)-i\sin(k_c\cdot x))\\&=2(a_1\cos(k_c\cdot x)-a_2\sin(k_c\cdot x)).
\end{align*}
where \((a_1,a_2)\) run along the circle \(a_1^2+a_2^2=r^2\).  Note that solutions depend solely on the two critical vectors of the lattice in which the magnitude is least, and that the spanning vectors play no direct role besides specifying the domain.

\appendix
\section{Dynamic Transition Theory}
In this appendix we recall some basic elements of the dynamic transition theory developed by Ma and Wang \cite{ptd}, which are used to carry out the dynamic transition analysis for the binary systems in this article. As mentioned in the introduction,  for many problems in sciences, we need to understand  the transitions  from one state to another, and the stability/robustness of the new states. For this purpose, a dynamic transition theory is developed Ma and Wang \cite{ptd}, and is applied to both equilibrium and non-equilibrium phase transitions in nonlinear sciences. The basic philosophy of the dynamic transition theory is to search for  the complete set of transition  states, which are represented by a  local attractor, rather than some steady states or periodic solutions or other type of orbits as part of this local attractor.  \\

Let $X$  and $ X_1$ be two Banach spaces,   and $X_1\subset X$ a compact and
dense inclusion. In this chapter, we always consider the following
nonlinear evolution equations
\begin{equation}
\left. 
\begin{aligned} 
&\frac{du}{dt}=L_{\lambda}u+G(u,\lambda),\\
&u(0)=\varphi ,
\end{aligned}
\right.\label{5.1}
\end{equation}
where $u:[0,\infty )\rightarrow X$ is unknown function,  and 
$\lambda\in \R^1$  is the system parameter.\\

Assume that $L_{\lambda}:X_1\rightarrow X$ is a parameterized
linear completely continuous field depending continuously on
$\lambda\in \R^1$, which satisfies
\begin{equation}
\left. 
\begin{aligned} 
&L_{\lambda}=-A+B_{\lambda}   && \text{a sectorial operator},\\
&A:X_1\rightarrow X   && \text{a linear homeomorphism},\\
&B_{\lambda}:X_1\rightarrow X&&  \text{a linear compact  operator}.
\end{aligned}
\right.\label{5.2}
\end{equation}
In this case, we can define the fractional order spaces
$X_{\sigma}$ for $\sigma\in \R^1$. Then we also assume that
$G(\cdot ,\lambda ):X_{\alpha}\rightarrow X$ is $C^r(r\geq 1)$
bounded mapping for some $0\leq\alpha <1$, depending continuously
on $\lambda\in \R^1$, and
\begin{equation}
G(u,\lambda )=o(\|u\|_{X_{\alpha}}),\ \ \ \ \forall\lambda\in
\R^1.\label{5.3}
\end{equation}\\

In addition, let $G(u,\lambda )$ have the Taylor
expansion at $u=0$ as follows
\begin{equation}
G(u,\lambda )=\sum^r_{m=k}G_m(u,\lambda )+o(\|u\|^r_{X_1})\qquad \text{ for  some }
2\leq k\leq r, \label{1.109}
\end{equation}
where $u\in X_1$, $G_m:X_1\times\cdots\times X_1\rightarrow X$ is an
$m$-multiple linear operator, and $G_m(u,\lambda
)=G_m(u,\cdots,u,\lambda )$.

Hereafter we always assume the conditions (\ref{5.2}) and
(\ref{5.3}), which represent that the system (\ref{5.1}) has
a dissipative structure.

Let the eigenvalues (counting multiplicity) of $L_{\lambda}$ be given by
$$\{\beta_j(\lambda )\in \C\ \   |\ \ j=1,2,\cdots\}$$
Assume that
\begin{align}
&  \text{Re}\ \beta_i(\lambda )
\left\{ 
 \begin{aligned} 
 &  <0 &&    \text{ if } \lambda  <\lambda_0,\\
& =0 &&      \text{ if } \lambda =\lambda_0,\\
& >0&&     \text{ if } \lambda >\lambda_0,
\end{aligned}
\right.   &&  \forall 1\leq i\leq m,  \label{a5.4}\\
&\text{Re}\ \beta_j(\lambda_0)<0 &&  \forall j\geq
m+1.\label{a5.5}
\end{align}\\

The following theorem is a basic principle of transitions; see \cite[Theorem 2.1.3]{ptd} for details. Basically this theorem  is the introduction of a dynamic classification scheme of dynamic transitions, with which  phase transitions, both equilibrium and non-equilibrium, are classified into three types: Type-I, Type-II and Type-III Mathematically, these transitions  are also respectively  c.           

\bt\cite[Theorem 2.1.3]{ptd}\la{tA.1}
 Let the conditions (\ref{a5.4}) and
(\ref{a5.5}) hold true. Then, the system (\ref{5.1}) must have a
transition from $(u,\lambda )=(0,\lambda_0)$, and there is a
neighborhood $U\subset X$ of $u=0$ such that the transition is one
of the following three types:

\begin{itemize}
\item[(1)] {\sc Continuous Transition}: 
there exists an open and dense set
$\widetilde{U}_{\lambda}\subset U$ such that for any
$\varphi\in\widetilde{U}_{\lambda}$,  the solution
$u_{\lambda}(t,\varphi )$ of (\ref{5.1}) satisfies
$$\lim\limits_{\lambda\rightarrow\lambda_0}\limsup_{t\rightarrow\infty}\|u_{\lambda}(t,\varphi
)\|_X=0.$$ In particular, the attractor bifurcation of (\ref{5.1})
at $(0,\lambda_0)$ is a continuous transition.

\item[(2)] {\sc Jump Transition}: 
for any $\lambda_0<\lambda <\lambda_0+\varepsilon$ with some $\varepsilon >0$, there is an open
and dense set $U_{\lambda}\subset U$ such that 
for any $\varphi\in U_{\lambda}$, 
$$\limsup_{t\rightarrow\infty}\|u_{\lambda}(t,\varphi
)\|_X\geq\delta >0,$$ 
where $\delta >0$ is independent of $\lambda$. 
This type of transition  is also called the discontinuous 
transition. 

\item[(3)] {\sc Mixed Transition}: 
for any $\lambda_0<\lambda <\lambda_0+\varepsilon$  with some $\varepsilon >0$, 
$U$ can be decomposed into two open sets
$U^{\lambda}_1$ and $U^{\lambda}_2$  ($U^{\lambda}_i$ not necessarily
connected):
$$\bar{U}=\bar{U}^{\lambda}_1+\bar{U}^{\lambda}_2,\ \ \
\ U^{\lambda}_1\cap U^{\lambda}_2=\emptyset ,$$ 
such that
\begin{align*}
&\lim\limits_{\lambda\rightarrow\lambda_0}\limsup_{t\rightarrow\infty}\|u(t,\varphi
)\|_X=0   &&   \forall\varphi\in U^{\lambda}_1,\\
& \limsup_{t\rightarrow\infty}\|u(t,\varphi
)\|_X\geq\delta >0 && \forall\varphi\in U^{\lambda}_2.
\end{align*}
%where  $U^{\lambda}_1$ is called the stable domain,  and  $U^{\lambda}_2$
%is the unstable domain.
\end{itemize}
\et

\section{Center Manifold reductions}
The stable and center manifold theorems are classical \cite{henry, b-book, ptd}. The main objective of this appendix is to provide a few important approximation formulas in Theorem~\ref{tB.1},  proved in \cite{b-book, ptd}. These formulas are used in the analysis of this paper.  \\

The spaces $X_1$ and $X$ can be decomposed into
the direct sum
$$
X_1=E_1^\lambda \bigoplus E_2^\lambda, \qquad X=E_1^\lambda\bigoplus\bar{E}_2^\lambda,
$$ 
where 
 \begin{align*} 
 & E_1^\lambda=\text{span}\ \{e_1(\lambda ),\cdots, e_m(\lambda )\},\\
& E_2^\lambda=\{u\in X_1|\ \  (u,e^*_i(\lambda ))=0 \quad  \forall 1\leq
i\leq m\},\\
& \bar{E}_2^\lambda=\text{closure\ of}\ E_2^\lambda \ \text{in}\ X.
\end{align*}
The linear operator $L_{\lambda}$ can be decomposed into
\begin{equation}
 L_{\lambda}=J_{\lambda}\bigoplus{\mathcal{L}}_{\lambda},
\label{1.112}
\end{equation}
where $J_{\lambda}: E_1^\lambda\rightarrow E_1^\lambda$ is the Jordan matrix of $L_{\lambda}$ at
$\beta_i(\lambda )$ $(1\leq i\leq m)$, and ${\mathcal{L}}_{\lambda}=L_{\lambda}|_{E_2^\lambda}:E_2^\lambda\rightarrow\bar{E}_2^\lambda$
possesses eigenvalues $\beta_j(\lambda )$   $(j\geq m+1)$. In this case
the equation (\ref{5.1}) can be written as
\begin{align}
& \frac{dx}{dt}=J_{\lambda}x+P_1G(x+y,\lambda ),\label{1.113}\\
&\frac{dy}{dt}={\mathcal{L}}_{\lambda}y+P_2G(x+y,\lambda
),\label{1.114}
\end{align}
where $P_i:X\rightarrow E_i^\lambda$   $(i=1,2)$ are the canonical projections.\\

The existence of  center manifold functions are classical; see among others \cite{henry, b-book}.
As the center manifold function is 
implicitly defined, one ingredient  in the theory and its applications of the phase transition dynamics in this book is that an approximation of the center manifold function to certain order will lead to a complete understanding of the transitions of a dynamical system. We now present  some  approximation formulas of the center manifold functions.\\

First, we show how the center manifold functions are constructed. 
Let
$\rho_{\varepsilon}:E^{\lambda}_1\rightarrow [0,1]$ be a $C^{\infty}$
cut-off function defined by
$$\rho_{\varepsilon}(x)=
\left\{
\begin{aligned}
& 1&&\text{if}\ \|x\|<\varepsilon ,\\
&0 &&\text{if}\ \|x\|>2\varepsilon, 
\end{aligned}
\right.$$ 
for some $\varepsilon >0$, and let
$$C^{0,1}(E^{\lambda}_1,E^{\lambda}_2(\theta
))=\{h:E^{\lambda}_1\rightarrow E^{\lambda}_2(\theta )|\ \ h(0)=0,
h\ \text{is\ Lipschitz}\}.$$

As in \cite{henry}, we need to find a
function $h\in C^{0,1}(E^{\lambda}_1,E^{\lambda}_2(\theta ))$
satisfying
\begin{equation}
h(\cdot)=\int^0_{-\infty}e^{-{\mathcal{L}}_{\lambda} \tau}\rho_{\varepsilon}(x(\tau
,\cdot ))G_2(x(\tau ,\cdot ) + h( x(\tau ,\cdot ) ), \lambda)d\tau,\label{1.106}
\end{equation}
where $x(t,x_0)$ is a solution of the ordinary differential
equation
\begin{equation}
\left. \begin{aligned}
&\frac{dx}{dt}=J_{\lambda} x+\rho_{\varepsilon}(x)G_1(x +h(x),\lambda
), && x(0)=x_0,
\end{aligned}
\right.\label{1.107}
\end{equation}
Then  the function $y(t,h(x_0))=h(x(t,x_0))$ is a
solution of the equation
$$\left. \begin{aligned}
&\frac{dy}{dt}={\mathcal{L}}_{\lambda} y+\rho_{\varepsilon}(x(t,x_0))G_2(x(t,x_0),y), && y(0)=h(x_0).
\end{aligned}
\right.$$ Thus, $x(t,x_0) + h(x(t,x_0))$ is a local solution of
(\ref{1.113})  and (\ref{1.114}), and the manifold $M_{\lambda}$ 
 is locally invariant for (\ref{5.1}).

The following theorem gives a first order approximation formula of
the center manifold function of (\ref{5.1}) near $\lambda
=\lambda_0$.

\bt  \cite[Theorem A.1.1]{ptd}\la{tB.1}
Assume that the nonlinear
operator $G(u,\lambda )$ has the Taylor expansion as (\ref{1.109}) at
$u=0$. Then the
center manifold function $\Phi :E_1\rightarrow E_2$ of (\ref{5.1}) near
$\lambda =\lambda_0$ can be expressed as
\begin{equation}
\Phi (x,\lambda
)=\int^0_{-\infty}e^{-\tau{\mathcal{L}}_{\lambda}}\rho_{\varepsilon}P_2G_k(e^{\tau
J_{\lambda}}x,\lambda )d\tau +o(\|x\|^k),\label{1.115}
\end{equation}
where $J_{\lambda}$ and ${\mathcal{L}}_{\lambda}$ are the linear
operators given by (\ref{1.112}), $G_k(u,\lambda )$ is the lowest order
$k$-multiple linear operator as in (\ref{1.109}), and
$x=\sum\limits^m_{i=1}x_ie_i\in E_1$. In particular, we have the
following assertions:

\begin{itemize}
\item[(1)] If $J_{\lambda}$ is diagonal near $\lambda =\lambda_0$, then
(\ref{1.115}) can be written as
\begin{equation}
-{\mathcal{L}}_{\lambda}\Phi (x,\lambda )=P_2G_k(x,\lambda
)+o(\|x\|^k)+O(|\beta |\|x\|^k),\label{1.116}
\end{equation}
where 
$\beta (\lambda )=(\beta_1(\lambda ),\cdots ,\beta_m(\lambda))$.

\item[(2)] Let $m=2$ and $\beta_1(\lambda )=\overline{\beta_2(\lambda
)}=\alpha (\lambda )+i\rho (\lambda )$ with $\rho (\lambda_0)\neq
0$. If $G_k(u,\lambda )=G_2(u,\lambda )$ is bilinear, then  
$\Phi (x,\lambda )$ can be expressed as
\begin{align}
&[ (-{\mathcal{L}}_{\lambda})^2+4\rho^2(\lambda)](-{\mathcal{L}}_{\lambda})\Phi (x,\lambda )
=[(-{\mathcal{L}}_{\lambda})^2+4\rho^2(\lambda )] P_2G_2(x,\lambda)  \nonumber  \\
& -2\rho^2(\lambda )P_2G_2(x,\lambda )+2\rho^2P_2G_2(x_1e_2-x_2e_1) \nonumber \\
& +\rho
(-{\mathcal{L}}_{\lambda})P_2[G_2(x_1e_1+x_2e_2,x_2e_1-x_1e_2)\nonumber\\
&+G_2(x_2e_1-x_1e_2,x_1e_1+x_2e_2)]+o(k),  \label{1.116-1}
\end{align}
where we have used 
$o(k)= o(\|x\|^k)+O(|\text{Re} \beta(\lambda) |\|x\|^k).$

\item[(3)] Let $\beta_1(\lambda)=\cdots = \beta_m(\lambda)$ have algebraic and geometric multiplicities $m\ge 2$  and $r =1 $ near $\lambda=\lambda_0$, i.e., $J_\lambda$  has the Jordan form:
\begin{equation}\label{A2.4.12-1}
J_\lambda=\left(
\begin{matrix}
\beta(\lambda) &   \delta                & \cdots     & 0                     &   0  \\
0                      &   \beta(\lambda)  & \cdots     & 0                     &   0  \\
\vdots               &  \vdots                & \vdots     & \vdots              &  \vdots \\
0                       & 0                        & \cdots     & \beta(\lambda) & \delta \\
0                       &   0                      & \cdots     & 0                       & \beta(\lambda)
\end{matrix}
\right)  \qquad \text{ for some }  \delta \not=0.
\end{equation}
Let 
$$z=\sum^m_{j=1} \xi_j e_j  \in E_1\quad \text{ with }  \quad \xi_j = \sum^{m-j}_{r=0}
\frac{\delta^r t^r x_{j+r}}{r!},$$
where $x=(x_1, \cdots x_m) \in \R^m$, $\delta \not= 0 $ is as in $J_\lambda$, and $t \ge 0$. 
Then the $k$-linear term $G_k(z, \lambda)$  can be expressed as 
$$G_k(z, \lambda) = F_1(x) + t F_2(x) + \cdots + t^{m-1} F_m(x),$$
and the center manifold function $\Phi$ can be expressed as 
\begin{equation}
\begin{aligned}
& \Phi= \sum^m_{j=1} \Phi_j + o(k), 
&&  -\mathcal L_\lambda^j \Phi_j = {(j-1)!} P_2 F_j(x) \quad \forall 1 \le j \le m.
\end{aligned}\la{A.2.4.12-2}
\end{equation}

\end{itemize}
\et

\section*{Acknowledgements}
The authors are grateful for two referees for their insightful comments. The work was supported in part by NSF grant DMS-2051032, and by Simons Foundation Travel Support for Mathematicians. Jared and Evan would also like to express our thanks to the Mathematics department of Indiana University for hosting the program.

\bibliographystyle{siam}

\bibliography{master}
\end{document}